\documentclass[ opre,nonblindrev]{informs3} % current default for manuscript submission

\OneAndAHalfSpacedXI % current default line spacing
\usepackage{endnotes}
\let\footnote=\endnote

\usepackage{natbib}
 \bibpunct[, ]{(}{)}{,}{a}{}{,}%
 \def\bibfont{\small}%
\TheoremsNumberedThrough     % Preferred (Theorem 1, Lemma 1, Theorem 2)
\ECRepeatTheorems

\EquationsNumberedThrough    % Default: (1), (2), ...
\usepackage{amsmath,amssymb} %amsthm
\usepackage{multirow}
\usepackage{authblk}
\usepackage{cite}
\usepackage{url}
\usepackage{pgf}
\usepackage{color}
\usepackage{booktabs} %for table weighted rules
\usepackage{algorithm}
\usepackage[noend]{algorithmic}
\usepackage{pgfplots}
\usepackage{caption}
\usepackage{soul}
\usepackage{bbm}
\usepackage{bm}
\usepackage{pdflscape}

\usepackage[absolute,overlay,showboxes]{textpos}
\usepackage{epstopdf}
\usepackage{subcaption}
\usepackage{changepage}
\usepackage{hyperref}

\usepackage{float}
\usepackage{tikz,pgfplots}
\usetikzlibrary{matrix}
\usepgfplotslibrary{groupplots}
\pgfplotsset{compat=newest}
\usetikzlibrary{calc,positioning,shapes,fit,arrows,shadows}
\tikzstyle{line} = [ draw, -latex']
\usetikzlibrary{shapes,arrows}
\usepgfplotslibrary{units}
\usepackage{url}
\usepackage{color}

\usetikzlibrary{shapes.geometric}
\usetikzlibrary{topaths,calc}

\renewcommand{\theARTICLETOP}{}
\renewcommand{\theLRHSecondLine}{}
\renewcommand{\theRRHSecondLine}{}
\renewcommand{\LRHFirstLine}{}
\renewcommand{\RRHFirstLine}{}
\renewcommand{\ECLRHFirstLine}{}
\renewcommand{\ECRRHFirstLine}{}

{\theoremheaderfont{THkey}\newtheorem{proof}{XXXXX}}

\usepackage{appendix}

\begin{document}
%%%%%%%%%%%%%%%%

% Outcomment only when entries are known. Otherwise leave as is and
%   default values will be used.
%\setcounter{page}{1}
%\VOLUME{00}%
%\NO{0}%
%\MONTH{Xxxxx}% (month or a similar seasonal id)
%\YEAR{0000}% e.g., 2005
%\FIRSTPAGE{000}%
%\LASTPAGE{000}%
%\SHORTYEAR{00}% shortened year (two-digit)
%\ISSUE{0000} %
%\LONGFIRSTPAGE{0001} %
%\DOI{10.1287/xxxx.0000.0000}%

% Author's names for the running heads
% Sample depending on the number of authors;
% \RUNAUTHOR{Jones}
% \RUNAUTHOR{Jones and Wilson}
% \RUNAUTHOR{Jones, Miller, and Wilson}
% \RUNAUTHOR{Jones et al.} % for four or more authors
% Enter authors following the given pattern:
\RUNAUTHOR{}

% Title or shortened title suitable for running heads. Sample:
% \RUNTITLE{Bundling Information Goods of Decreasing Value}
% Enter the (shortened) title:
\RUNTITLE{}

% Full title. Sample:
% \TITLE{Bundling Information Goods of Decreasing Value}
% Enter the full title:
\TITLE{A Joint Chance-Constrained \\ Stochastic Programming Approach for the  Integrated Predictive Maintenance and Operations \\ Scheduling Problem in Power Systems}

% Block of authors and their affiliations starts here:
% NOTE: Authors with same affiliation, if the order of authors allows,
%   should be entered in ONE field, separated by a comma.
%   \EMAIL field can be repeated if more than one author
\ARTICLEAUTHORS{%
\AUTHOR{Bahar Cennet Okumuşoğlu}
\AFF{Industrial Engineering Program, Sabancı University, 34956 Istanbul, Turkey, \EMAIL{\href{mailto:okumusoglu@sabanciuniv.edu}{okumusoglu@sabanciuniv.edu}}} %, \URL{}}
\AUTHOR{Beste Basciftci}
\AFF{Department of Business Analytics, University of Iowa, Iowa City, IA, 52242, USA, \EMAIL{\href{mailto:beste-basciftci@uiowa.edu}{beste-basciftci@uiowa.edu}}}
\AUTHOR{Burak Kocuk}
\AFF{Industrial Engineering Program, Sabancı University, 34956 Istanbul, Turkey, \EMAIL{\href{mailto:burak.kocuk@sabanciuniv.edu}{burak.kocuk@sabanciuniv.edu}}}
% Enter all authors
}

\ABSTRACT{Maintenance planning plays a key role in power system operations under uncertainty as it helps system operators ensure a reliable and secure power grid. This paper studies a short-term condition-based integrated maintenance planning with operations scheduling problem while considering the unexpected failure possibilities of generators as well as transmission lines. We formulate this problem as a two-stage stochastic mixed-integer program with failure scenarios sampled from the sensor-driven remaining lifetime distributions of the individual system elements whereas a joint chance-constraint consisting of Poisson Binomial random variables is introduced to account for failure risks. 
%The first-stage of this stochastic program focuses on the maintenance scheduling problem, whereas each second-stage subproblem corresponds to a unit commitment problem under a failure scenario. 
Because of its intractability, we develop a cutting-plane method to obtain an exact reformulation of the joint chance-constraint by proposing a separation subroutine and deriving stronger cuts as part of this procedure. To solve large-scale instances, we derive a second-order cone programming based safe approximation of this constraint. Furthermore, we propose a decomposition-based algorithm implemented in parallel fashion for solving the resulting stochastic program, which exploits the features of the integer L-shaped method and the special structure of the maintenance and operations scheduling problem to derive stronger optimality cuts. We further present preprocessing steps over transmission line flow constraints to identify redundancies. To illustrate the computational performance and efficiency of our algorithm compared to more conventional maintenance approaches, we design a computational study focusing on a weekly plan with daily maintenance and hourly operational decisions involving detailed unit commitment subproblems. Our computational results on various IEEE instances demonstrate the computational efficiency of the proposed approach with reliable and cost-effective maintenance and operational schedules.}


\maketitle
%%%%%%%%%%%%%%%%%%%%%%%%%%%%%%%%%%%%%%%%%%%%%%%%%%%%%%%%%%%%%%%%%%%%%%

% Samples of sectioning (and labeling) in OPRE
% NOTE: (1) \section and \subsection do NOT end with a period
%       (2) \subsubsection and lower need end punctuation
%       (3) capitalization is as shown (title style).
%
%\section{Introduction.}\label{intro} %%1.
%\subsection{Duality and the Classical EOQ Problem.}\label{class-EOQ} %% 1.1.
%\subsection{Outline.}\label{outline1} %% 1.2.
%\subsubsection{Cyclic Schedules for the General Deterministic SMDP.}
%  \label{cyclic-schedules} %% 1.2.1
%\section{Problem Description.}\label{problemdescription} %% 2.

% Text of your paper here
\section{Introduction}
%As the world’s electricity consumption continuously grows,
%With the continuous growth in the world's electricity consumption, 
The competitive power industry has challenged the system operators with prohibitive penalty costs to continue their operations uninterruptedly. %due to unexpected failures of system components. 
A natural way to avoid such interruptions with these penalties is scheduling maintenance for the system components while leveraging their condition information. Such condition-based maintenance increases the operational lifetime of the aging power grid; however, ignoring power system capabilities when performing maintenance may cause large-scale blackouts resulting in additional maintenance and operational costs (see, for example Florida blackout in 2008 \citep{FRCC}). Thus, the condition-based maintenance schedules for generators as well as transmission lines must be coordinated with operational schedules in order to preserve the security of the power grid.

Maintenance schedules of generators and transmission lines have a great effect on power generation as well as power flow. Still, obtaining optimal maintenance schedules for generators has aroused considerably more interest than for transmission lines in the power system literature (\citet{Conejo2005}, \citet{Wu2008}, \citet{Canto2008}). In transmission maintenance planning problem, mathematical complexity ensues from the removal of transmission lines for their unavailable periods due to maintenance, which results in a change in the network topology. Besides, such removals may cause congestion in the power system affecting the reliability and the security of the system. In this respect, the joint optimization of generator and transmission line maintenance (hereafter referred to as the {\it integrated maintenance}) planning problem becomes more critical in power systems to ensure reliable system operations by capturing the complex nature of the problem. %As regards to the respective planning horizon, the integrated short-term maintenance will help system operators ensure an equitable coordination between the market players in a reliable power network.

In the competitive power industry, cost-effective maintenance schedules and demand-delivery under failure uncertainty have become more and more important. Recent advances in grid modernization such as condition-monitoring are widely employed to deal with this failure uncertainty (\citet{Yildirim2016-1}, \citet{Basciftci2018}). In condition-monitoring systems, sensors connected to the power grid monitor the emerging health conditions of degrading system components. These systems can be used as a basis for estimating the residual lifetime of the components by means of degradation signals obtained from real-time sensor information. When scheduling maintenance, such condition-based information on the underlying uncertainty avails system operators of correcting natural causes from degradation and increasing the overall operational lifetime of the aging power infrastructure.
% reducing the failure probability of system components and consequently, 

Many optimization problems in power systems can be modeled as large-scale stochastic mixed-integer programs (SMIPs) as they involve various uncertainties and risks as well as a vast number of binary variables related to  maintenance schedules, commitment status of generators and switching status of transmission lines. 
%For this matter, stochastic programming has attracted many researchers in power systems. 
In view of handling uncertainties, building SMIP models with scenario-dependent variables and constraints is the most prevalent approach. The SMIP models, even with a limited number of scenarios, may become computationally demanding, and moreover, solutions for these SMIPs given by the state-of-the-art solvers can be suboptimal. Thus, these SMIP models necessitate developing novel decomposition-based solution algorithms to achieve tractability. To handle risks, on the other hand, the SMIP models can be built with chance-constraints. There are only a limited number of cases where a chance-constraint is computationally tractable and whenever this is not the case, it can be replaced with its safe approximation which imposes conservatism on the underlying problem. Thus, it becomes critical to provide an equivalent description of these chance-constraints whenever possible within the modeling process. 

In this paper, we study an integrated short-term condition-based maintenance scheduling problem in coordination with operations planning by taking account of unexpected failures of generators as well as transmission lines. 
%We extend the stochastic optimization model of \citet{Basciftci2018} by considering unexpected failures of transmission lines, switching decisions, and direct current power flow equations and model the resulting problem as a joint chance-constrained SMIP. 
We explicitly depict the underlying failure uncertainty as a continuous stochastic degradation process and utilize sensor-driven real-time information to estimate the remaining lifetime distribution (RLD) of system components. Furthermore, we identify those system components prone to failure within the planning horizon and construct failure scenarios based on their estimated RLDs. Additionally, we propose a joint chance-constraint for simultaneously restricting the total number of corrective maintenance occurring due to unexpected failures for generators and transmission lines within the planning horizon along with its exact and safe representation approaches. We develop a decomposition-based cutting-plane framework to efficiently solve the resulting large-scale problem and obtain optimal condition-based daily maintenance schedules, and hourly operational decisions. We validate these maintenance schedules by evaluating them over a larger size of failure scenarios over all system components under a sample average approximation (SAA) approach. 

Our paper makes the following contributions:
\begin{itemize}
    \item We develop a stochastic optimization framework which combines the short-term condition-based generator and transmission line maintenance, and operations planning problems while explicitly considering the impacts of unexpected failures of generators as well as transmission lines on power system operations. Our framework differs from the existing studies in considering condition-based transmission line maintenance with generator maintenance. By engaging real-time degradation-based sensor information in the elaborate failure characterization under a Bayesian approach, we predict the RLDs of generators and transmission lines, and identify a specific subset of these power system components prone to failure within the planning horizon.
    
    \item To account for the failure uncertainty of both generators and transmission lines in our stochastic optimization model, we generate failure scenarios based on their estimated RLDs. We also introduce  a joint chance-constraint to mitigate the unexpected failure risks for generators and transmission lines. Because of its intractability, we develop a cutting-plane method to obtain an exact reformulation of the joint chance-constraint through a separation subroutine and a set of improved cuts. Our solution framework leverages Poisson Binomial random variables in this joint chance-constraint, which can be extended to the settings under similar forms. Moreover, we derive a second-order cone programming based safe approximation of this constraint.
	
	\item We develop a decomposition-based algorithm by improving the integer L-shaped method with various algorithmic enhancements. We exploit the nice and special structure of the scenario subproblems and introduce two concepts: time-decomposability and status of system components. We benefit from these concepts to decrease the total number of scenario subproblems solved and moreover, generate various sets of stronger optimality cuts than the integer L-shaped optimality cuts.  
	We employ parallel computing to implement our decomposition algorithm more efficiently and further present preprocessing steps for identifying redundant transmission line constraints. 
	
	\item We conduct a computational study with various modified IEEE instances to illustrate the computational efficiency of each algorithmic enhancement. We also compare the proposed decomposition algorithm with an existing state-of-the-art solver. For all instances, the underlying problem can be solved orders of magnitude faster with the proposed decomposition algorithm than this solver. Our computational study also shows that the proposed stochastic framework provides $14-31 \%$ cost savings using both the exact reformulation and safe approximation of the joint chance-constraint in comparison with the deterministic model.
\end{itemize}

The remainder of our paper is organized as follows. In Section \ref{s:literatureReview}, we review the relevant literature. In Section \ref{s:stochasticOptimizationModel}, we describe the integrated short-term condition-based maintenance scheduling with operations planning problem, the degradation signal modeling and decomposition structure in detail. The solution methodology with various algorithmic enhancements is presented in Section \ref{s:solutionMethodology}. The computational experiments and extensive numerical results follow in Section \ref{s:computationalExperiments}. We conclude our paper with final remarks in Section \ref{s:conclusions}.

%This paper contributes to the existing literature by engaging real-time degradation-based sensor information in the elaborate failure characterization of power system components under a Bayesian approach, developing a stochastic optimization framework to solve the integrated short-term condition-based maintenance planning problem with operations scheduling problems as well as explicitly considering the impacts of the unexpected failures of system components on power system operations. 
\section{Literature Review} \label{s:literatureReview}

In this section, we review the relevant power system literature on maintenance planning problem (Section \ref{ss:maintenancePlanningPowerSystems}), failure uncertainty (Section \ref{ss:failureUncertaintyInPowerSystems}) and stochastic programming (Section \ref{ss:stochasticProgrammingPowerSystems}). We explicate the contributions of our paper in each section.

\subsection{Maintenance Planning in Power Systems} \label{ss:maintenancePlanningPowerSystems}
Maintenance planning problem in power systems has been widely studied in the literature (for a recent review, see \citet{Froger2016}). This problem concerns both generators and transmission lines, and ideally attempts to identify the unavailability of these components while ensuring a reliable power grid. However, the majority of the existing studies have focused more on obtaining optimal maintenance schedules for generators subject to various operational and network constraints (\citet{Conejo2005}, \citet{Wu2008}, \citet{Canto2008}, \citet{Yildirim2016-2}, \citet{Basciftci2018}) than for transmission lines (\citet{Marwali2000}, \citet{Abbasi2009}, \citet{Abiri2009}, \citet{Conejo2012}, \citet{Lv2012}). This is because of the fact that the power network topology will inherently change due to the maintenance actions for transmission lines, and this varying network topology exceedingly influences the power generation and further complicates the resulting problem.
%in power systems. 
%\citet{Conejo2012} propose a bilevel model for transmission maintenance with power flow constraints over a yearly time horizon. The authors also employ the strong duality to reformulate this problem as a single-level model. However, dual information may not be available when solving transmission maintenance planning problem in a large-scale power network since some operational decisions such as switching status of transmission lines are binary in nature. Some studies do not benefit from the duality theory and use a priority-based dynamic programming \citep{Abbasi2009} to obtain long-term maintenance schedules for overhead transmission lines. However, dynamic programming approach may not be applied in real-time power system operations due to the phenomenon referred to as the curse of dimensionality.
% However, only a few studies address the

The integrated maintenance problem can yield more cost-effective maintenance and operational schedules; however, another source of complexity arises when coordinating the maintenance schedules for both generators and transmission lines. Therefore, the integrated maintenance planning problem has attracted only very few researchers in power systems. The coordination of generator and transmission line maintenance schedules coupled with the security constrained UC is analyzed by \citet{Fu2007} and \citet{Fu2009}. Optimization models in these works can be utilized both in vertically integrated and restructured power systems.
\citet{Geetha2009} coordinate integrated maintenance schedules with an acceptable level of reliability between independent actors in restructured power systems. These studies do not account for the uncertainty resulting from the unexpected failures of system components, which strongly affects the maintenance planning problem. The study by \citet{Wang2016N1} models the generators and transmission maintenance planning problem incorporating N-1 security criterion; however, this deterministic model may provide infeasible maintenance schedules when multiple failures of system components occur in the power system. \citet{Wu2010} formulate an integrated maintenance problem in coordination with the security-constrained UC considering various uncertainties including forced outages of generators and transmission lines over a long-term planning horizon. They model these forced outage rates as a Markov process by using predetermined constant failure characteristics, which may not be a realistic assumption in a dynamic power network. \citet{Wang2016} propose a similar approach to jointly optimize the underlying problem with the security-constrained UC by updating outage scenarios in an iterative manner. However, the authors neglect to consider the effects of these scenarios on operations planning and do not leverage sensor-driven condition information to identify critical system components prone to failure, which is imperative for securing overall power system operations. %Meta-heuristic approaches are also applied to determine optimal maintenance schedules for generator and transmission lines such as particle swarm optimization \citep{Abirami2014} and fuzzy evolutionary algorithm \citep{Sharkh2003}. Still, maintenance schedules given by these methods may not be applicable in practice as they fail to certify the optimality of their solutions. 

\subsection{Failure Uncertainty in Power Systems} \label{ss:failureUncertaintyInPowerSystems}
Quantifying the failure uncertainty in power system operations has been instrumental in the maintenance planning problem. In order to achieve cost-effective maintenance schedules and to ensure the reliability and the security of the aging power infrastructure, the stochasticity arising from failures of system components must be considered in real-time operations. Power system components depict degradation symptoms over time from increasing wear and tear. This continuous degradation process may eventually lead to unexpected failures resulting in unscheduled shutdowns, congested transmission lines, voltage instability and sudden increase in power demand. To extenuate the disruptive impacts of the failures of these system components, many existing operational strategies such as N-1 contingency criterion \citep{N-1criterion} and reserve requirements, and maintenance strategies such as periodic and manufacturer-recommended maintenance schedules \citep{marketOperationsShahidehpour} are used in power systems. These deterministic strategies remain as half measures and are not enough to improve the utilization of the power grid and therefore, many energy companies have recently started to adapt condition-monitoring techniques because of their potential benefits (for a comprehensive review, see \citet{Han-Song2003}). In particular, these are widely employed to estimate the RLDs of system components by tracking degradation of these components using sensors in order to account for unexpected failures.
%Monitoring degradation process of system components to detect inceptive or abrupt failures without the proliferation of condition monitoring systems would require strenuous effort and, as a result, they are widely employed to account for these failures of system components by estimating their RLDs.

Although the failure uncertainty of system components has been considered in the literature for modeling power system operations, most studies neglect component-specific condition information and further assume that system components carry constant failure characteristics over the planning horizon (\citet{Wu2010}, \citet{Papa2013}). As this approach becomes insufficient in capturing the condition information of the components, a few studies recently consider the underlying failure uncertainty by incorporating degradation-based approaches. \citet{Wang2016} adopt a degradation-based model by extending the traditional hazard model and dynamically updating failure characteristics of system components. %; however, the authors do not incorporate various unexpected failure scenarios into the power system operations. 
Furthermore, a deterministic mixed-integer optimization model integrated with condition-based sensor information is presented to obtain optimal maintenance schedules for generators \citep{Yildirim2016-1, Yildirim2016-2}. Recently, \citet{Basciftci2020} propose a similar framework by leveraging time-varying load-dependency to obtain condition-based maintenance schedules for a fleet of generators by presenting a decision-dependent stochastic program to capture the RLDs of the components depending on the operational decisions. Nevertheless, these studies either have been conducted in a deterministic fashion and/or do not take into account scenario-dependent failure uncertainty for both generators and transmission lines at the same time. The optimization framework proposed by \citet{Basciftci2018} embodies sensor-driven condition-based information in the long-term generator maintenance and operations planning problem considering only failure scenarios of generators; however, the authors do not consider the failure uncertainty of transmission lines and their effects on power system operations.
%Although generator maintenance planning problem through degradation-based sensor information has been studied previously,
The existing literature still lacks a unified framework for addressing the integrated condition-based maintenance planning by considering the impacts of the sensor-driven failure uncertainty of both generators and transmission lines on power system operations. As this unified framework becomes critical in ensuring cost-effective and reliable operations of the power systems, the complexities arisen from the integration of line maintenance decisions and their failure possibilities need to be addressed by developing various stochastic optimization techniques, which consists a significant part of the contributions of this study, that can also be extended to problem settings with similar structure. %including derivations for alternative representations of joint probabilistic constraints and tailoring solution algorithms by leveraging problem properties, as part of the contributions of this study. 

%One of the main contributions of our paper is developing a stochastic optimization model for the integrated short-term condition-based maintenance planning problem in coordination with power system operations. We model the underlying failure uncertainty by means of degradation signal models and update the RLDs of system components by using real-time sensor information under a Bayesian setting. 
%Because of the fact that considering all system components under maintenance is impractical in real-time power system operations, a subset of these system components is elaborately selected based on their conditions in Section \ref{ss:problemSetting}. For this specific subset, we generate various failure scenarios based on their estimated RLDs and account for their impacts on power system operations in the stochastic optimization model. We present the resulting stochastic optimization model which incorporates failure uncertainty of each individual system component through degradation signal models in Section \ref{ss:StochasticOptimizationModel}.

\subsection{Stochastic Programming in Power Systems} \label{ss:stochasticProgrammingPowerSystems}
Stochastic programming arises as an important tool for modeling power system operations under uncertainty. Many existing studies in the literature describe the underlying uncertainties with a set of scenarios, i.e., a set of possible realizations of random variables (\citet{Wu2008}, \citet{Papa2013}, \citet{Papa2015}, \citet{Basciftci2018}). Still, conventional methods may not be sufficient to solve the resulting problem in a reasonable amount of time as the scenario set can consist of an extremely large number of scenarios. As this set grows exponentially fast in the size of the network components, solving large-scale problems in power systems necessitates specialized decomposition techniques. Fortunately, the majority of such large-scale problems in power systems can be intrinsically decoupled into many smaller problems, and then recast as two-stage stochastic programs under suitable conditions.
%are amenable to decomposition Various novel methods considering this special structure of the two-stage stochastic programs have been introduced in the literature. 
%  The inefficiency of the cutting plane schemes when degeneracy occurs has motivated \citet{Rusz1986} to propose regularized decomposition.
\citet{SlykeRoger1969} introduced the continuous L-shaped method %, also known as Benders' decomposition,
as a cutting plane technique to solve the two-stage stochastic linear programs with recourse. A common criticism for this method is that the linear programming duality cannot be readily applied when integer decisions exist in the second-stage problems. In particular, SMIPs are known to have their combinatorial challenges attributed to the non-convex (even discontinuous) nature of the expected second-stage objective function. The integer L-shaped method, proposed by  \citet{Laporte1993}, can be applied to solve the two-stage mixed-integer stochastic programs with pure binary first-stage decisions and mixed-integer second-stage decisions. As this algorithm can be extendable to our problem setting, we propose a decomposition-based algorithm in Section \ref{ss:decompositionAlgorithm} by using the features of the integer L-shaped method to solve the integrated short-term condition-based maintenance scheduling with operations planning problem. 
%We formulate this problem as a two-stage joint chance-constrained SMIP with failure scenarios sampled from the sensor-driven remaining lifetime distributions of the individual system elements. The first-stage of this stochastic program focuses on the maintenance planning problem. Operational decisions associated with the unit commitment and economic dispatch problems, called second-stage decisions, are taken after the failure uncertainty is disclosed. 
%Furthermore, we exploit the special structure of the resulting problem and provide an algorithmic enhancement which significantly decreases the computational effort required to solve the second-stage problems. In Section \ref{ss:newOptimalityCuts} and Section \ref{ss:alternativeOptimalityCuts} under this special structure, we derive valid and stronger optimality cuts than the integer L-shaped optimality cuts, which are integrated into the solution procedure implemented in a parallel fashion. 
By exploiting the special structure of this problem, we provide algorithmic enhancements which significantly decreases the computational effort required to solve the second-stage problems, and derive stronger optimality cuts than the integer L-shaped optimality cuts, which are integrated into our solution procedure implemented in a parallel fashion in Section \ref{ss:newOptimalityCuts}.

Chance constraints are widely used in modeling power systems operations as they are subject to various risks associated with many uncertainties (for a comprehensive review, see \citet{Geng2019}). Although chance-constraints have great importance for mitigating risks in power system operations, the feasible set defined by a chance-constraint is in general nonconvex, and obtaining an exact representation of such a constraint can be difficult even under the assumption of convexity. A very well-known case in which such issues do not appear is when a random variable associated with the chance-constraint follows a Gaussian distribution and the probability level of the chance constraint is at least $0.5$. In this case, the corresponding feasible set can be represented as a second-order conic set \citep{Nemirovski2012}. Many studies in power systems follow this Gaussian assumption and obtain such deterministic equivalents of the chance-constraints (\citet{Wu2014}, \citet{Roald2017}). In practice, it may happen that the probability distribution of the random variable is not Gaussian and such tractable representations may not be readily available. Whenever this is the case, safe approximations can be obtained as an alternative, though conservative, representations of the chance-constraints. Recently in maintenance planning literature, \citet{Basciftci2018} introduce a single chance-constraint consisting of the sum of independent Bernoulli random variables, i.e., a Poisson Binomial random variable. The authors ignore this useful information on the underlying distribution and propose a deterministic safe approximation of the chance-constraint by using Markov and Bernstein bounds. Joint chance-constraints are relatively more difficult to handle than a single chance-constraint. Many studies in the literature reformulate the feasible set of the joint chance-constraint by using Bonferroni-based safe approximation (\citet{Ozturk2004}, \citet{Peng2013}, \citet{Baker2017}); however, this safe approximation is likely to be overly conservative.
Thus, for the chance-constraints, there is a trade-off between searching for exact reformulations or deriving safe approximations to provide their alternative representations. 

To mitigate failure risks of generators and transmission lines, we introduce a joint chance-constraint which restricts the total number of these system components under corrective maintenance, which is an undesirable and costly maintenance in case of an unexpected failure. Our joint chance-constraint consists of Poisson Binomial random variables by leveraging the RLDs of the system components. In contrast to the recent work by \citet{Basciftci2018} with a single chance-constraint, we exploit the underlying distribution and propose an exact reformulation of the joint chance-constraint in  Section \ref{ss:exactReformulation}.
Our proposed decomposition algorithm under exact reformulation can be used for any two-stage joint chance-constrained stochastic program with pure binary first-stage decisions and independent Poisson Binomial random variables associated with this joint chance-constraint. Further, we investigate the separation problem over the joint chance-constraint and develop a separation subroutine within our decomposition algorithm.  By exploiting the distributional information on the Poisson Binomial random variables, we strengthen the cutting planes which are generated within the separation subroutine.
%We propose an exact reformulation of the joint chance-constraint under the assumption that a probability oracle exists and computes the exact value of the probability of the non-convex joint chance-constraint by using the knowledge on the distribution of the random variables in Section \ref{ss:exactReformulation}.
%Further, we investigate the separation problem over the joint chance-constraint and use this probability oracle as a separation subroutine within our proposed decomposition algorithm to check the feasibility of a given solution and generate violated cover inequalities, if such an equality exists.
%This cutting-plane method guarantees an exact solution but may show slow convergence and require more computational effort as the size of the problem increases. 
To solve large-scale instances, we also propose a second-order cone programming based safe approximation of the joint chance-constraint in Section \ref{ss:deterministicSafeApproximation}. Without any assumption on the underlying distribution of the independent random variables associated with the joint chance-constraint, our decomposition algorithm under safe approximation can also be extended to handle any two-stage joint chance-constrained stochastic program with pure binary first-stage decisions.
%The proposed decomposition algorithm under exact reformulation and safe approximation can be extended to handle any SMIPs with joint-chance constraint consisting of independent Poisson Binomial random variables.
%The key basis of this reformulation is a second-order cone programming reformulation by lifting the associated feasible set to a higher-dimensional space by introducing auxiliary variables.

\section{Stochastic Optimization Model} \label{s:stochasticOptimizationModel}
In this section, we first describe the problem setting (Section \ref{ss:problemSetting}) and present the joint chance-constrained stochastic optimization model (Section \ref{ss:StochasticOptimizationModel}). We explain how to characterize the underlying failure uncertainty by using degradation signal modeling in detail in Section \ref{ss:DegradationSignalModelingandScenarioGeneration}. We provide the compact formulation and decomposition-based reformulation of our optimization model in Section \ref{ss:DecompositionStochOPT}.

\subsection{Problem Setting} \label{ss:problemSetting}
In our study, we consider a power network $\mathcal{N} = (\mathcal{B}, \mathcal{L})$, where $\mathcal{B}$ and $\mathcal{L}$ represent the sets of buses and transmission lines, respectively.  We denote the set of generators linked to buses as $\mathcal{G} \subseteq \mathcal{B}$. In particular, $\mathcal{G}(i)$ denotes the set of generators attached to bus $i$. We let $\delta^+(i)$ and $\delta^-(i)$ be the sets of outgoing and incoming neighbors of bus $i$, respectively. We define $\mathcal{G'}$ as the set of generators which potentially need to be maintained, and  $\mathcal{G''}$ as the set of generators which are not scheduled for maintenance within the planning horizon due to their low failure probabilities as detailed below. Similarly, we define the sets $\mathcal{L'}$ and $\mathcal{L''}$ for representing the transmission lines. In the remainder of this paper, we use the term ``component'' to refer both generators and transmission lines and let the set of components to be $\mathcal{H} = \mathcal{H}' \cup \mathcal{H}''$, where $\mathcal{H}' = \mathcal{G}' \cup \mathcal{L}'$ and $\mathcal{H}'' =\mathcal{G}'' \cup \mathcal{L}''$. We explicitly specify the type of components with subscripts when necessary. 
%Along with the maintenance and operations scheduling decisions of the components, transmission switching decisions for the set of lines $\mathcal{L}'$ are considered within planning.  

The proposed stochastic optimization model incorporates the uncertainty in failure times of system components. In addition to the introduced joint chance-constraint that ensures the reliable operations of the system based on RLDs, we represent the  uncertainty in the optimization model with a finite set of scenarios, denoted by $\mathcal{K}$, where scenario $k$ contains a possible realization of random failure time $\xi^k_h$ of component $h$. We also consider a finite set of maintenance periods, denoted by $\mathcal{T}$, and a finite set of hourly subperiods in each maintenance period, denoted by $\mathcal{S}$. Additionally, we define an extended planning horizon as $\mathcal{ \bar T} = \mathcal{T} \cup \{{ |\mathcal{T}| + 1}\}$ for cases in which components do not fail within the planning horizon. We identify subset $\mathcal{H}'$ based on the RLDs of system components. The main reason of this subset selection is that scheduling all system components for short-term maintenance is impractical and unnecessary in real-time power system operations. 
%Thus, identifying subset $\mathcal{H'}$ needs meticulous and special selection procedure, which we apply using remaining lifetime distributions. 
We explain how to obtain a characterization on the RLDs in detail in  Section \ref{ss:DegradationSignalModelingandScenarioGeneration}. Here, we outline the main steps for identifying subset $\mathcal{H}'$. Suppose we are given a probability threshold $\bar p_{fail} \in [0,1]$ (e.g., $\bar p_{fail} = 0.1$). For component $h \in \mathcal{H}$, we first obtain its failure probability $p^h_{fail}$ within the planning horizon. If $p^h_{fail} \ge \bar p_{fail}$, we add component $h$ to set $\mathcal{H}'$.
After having identified the set $\mathcal{H}'$, we sample failure scenarios for each component $h \in \mathcal{H}'$ from its unique RLD based on the scenario generation procedure proposed by \citet{Basciftci2018}. If a component $h$ does not fail within the planning horizon under scenario $k$, we let $\xi_{h}^k = |\mathcal{\bar T}|$. 
We note that the components in $\mathcal{H}'$ are assumed to enter maintenance at most once, whereas  components belonging to set $\mathcal{H}''$ are not scheduled for maintenance within the planning horizon since their failure probabilities are negligible. In our solution evaluation scheme, however, we assume that all components from set $\mathcal{H}$ may fail within the planning horizon.

\subsection{Mathematical Model and Formulation} \label{ss:StochasticOptimizationModel}
In this section, we first introduce the necessary notations for our optimization model. In Table \ref{fig:problemParameters}, we present the notation used for the decision variables and parameters along with their definitions. The scenario-dependent decisions variables (also parameters) are associated with the superscript~$k$. %The scenario-dependent decision variables are taken after the failure uncertainty is disclosed.
%
%
% Decision variables & parameters
%
%
\begin{table}[h]
\centering
\begin{tabular}{|rl}
\hline
\multicolumn{2}{|l|}{$\textbf{Parameters}$}   \\ \hline
$\pi^k$                        & \multicolumn{1}{l|}{Probability of scenario $k$.}                                              \\
$\xi^k_i$                      & \multicolumn{1}{l|}{Failure time of generator $i$ in scenario $k$.}                             \\
$\xi^k_{ij}$                   & \multicolumn{1}{l|}{Failure time of transmission line $(i,j)$ in scenario $k$.}                             \\
$\tau^{p}_{{\scriptscriptstyle \mathcal{G}}} (\tau^{c}_{{\scriptscriptstyle \mathcal{G}}})$                         & \multicolumn{1}{l|}{Predictive (corrective) maintenance duration of generators.}                     \\
%$\tau^{c}_{{\scriptscriptstyle \mathcal{G}}}$                          & \multicolumn{1}{l|}{Corrective generator maintenance duration in periods.}                     \\
$\tau^{p}_{{\scriptscriptstyle \mathcal{L}}} (\tau^{c}_{{\scriptscriptstyle \mathcal{L}}}) $                     & \multicolumn{1}{l|}{Predictive (corrective) maintenance duration of transmission lines.}                          \\
%$\tau^{c}_{{\scriptscriptstyle \mathcal{L}}}$                          & \multicolumn{1}{l|}{Corrective transmission line maintenance duration in periods.}                          \\
$C^{p}_{i}(C^{c}_i)$                      & \multicolumn{1}{l|}{Predictive (corrective) maintenance cost of generator $i$ in period $t$.}                                    \\
%$C^{c}_i$                      & \multicolumn{1}{l|}{Corrective maintenance cost for generator $i$.}                                    \\
$C^{p}_{ij} ( C^{c}_{ij})$                      & \multicolumn{1}{l|}{Predictive (corrective) maintenance cost of transmission line $(i,j)$ in period $t$.}                            \\
%$C^{c}_{ij}$                      & \multicolumn{1}{l|}{Corrective maintenance cost for transmission line $(i,j)$.}                            \\
$C_{i}^g$                        & \multicolumn{1}{l|}{Generation cost of generator $i$.}                                    \\
$C_i^n$ & \multicolumn{1}{l|}{No-load cost of generator $i$.} \\
$C_{i}^s$  & \multicolumn{1}{l|}{Start-up cost of generator $i$.}\\
$C^{d}_i$  & \multicolumn{1}{l|}{Demand curtailment cost of generator $i$.}\\
${\delta}^{\max}_i ({\delta}^{\min}_i)$                & \multicolumn{1}{l|}{Maximum (minimum) voltage angle at bus $i$.}                                         \\
%${\delta}^{\min}_i$ & \multicolumn{1}{l|}{Minimum voltage angle at bus $i$.}                                         \\
$p^{\max}_i(p_i^{\min})$                     & \multicolumn{1}{l|}{Maximum (minimum) power generation of generator $i$.}                                \\
%$p_i^{\min}$        & \multicolumn{1}{l|}{Minimum power generation at generator $i$.}                                \\
$MU_i (MD_i)$                         & \multicolumn{1}{l|}{Minimum up (down) time of generator $i$.}                                         \\
%$MD_i$                         & \multicolumn{1}{l|}{Minimum down time of generator $i$.}                                       \\
$RU_i (RD_i)$                         & \multicolumn{1}{l|}{Ramp up (down) rate of generator $i$.}                                            \\
%$RD_i$                         & \multicolumn{1}{l|}{Ramp-down rate of generator $i$.}                                          \\
$B_{ij}$                       & \multicolumn{1}{l|}{Susceptance of transmission line $(i,j).$}                                               \\
$d_{its}$                      & \multicolumn{1}{l|}{Power demand of bus $i$ in operational subperiod $s$ of period $t$.} \\
$M_{ij}$                       & \multicolumn{1}{l|}{Sufficiently large number for a flow constraint of transmission line $(i,j)$.}                               \\
\hline
\multicolumn{2}{|l|}{$\textbf{Decision Variables}$}     \\ \hline
$w_{it}$           & \multicolumn{1}{l|}{$1$ if generator $i$ enters maintenance in period $t$, and $0$ otherwise.}                        \\
$z_{ijt}$          & \multicolumn{1}{l|}{$1$ if transmission line $(i,j)$ enters maintenance in period $t$, and $0$ otherwise.}                         \\
$\delta_{its}^k$   & \multicolumn{1}{l|}{Voltage angle at bus $i$ in subperiod $s$ of period $t$ in scenario $k$.}                         \\
$q^k_{its}$        & \multicolumn{1}{l|}{Demand curtailed at bus $i$ in subperiod $s$ of period $t$ in scenario $k$.}                      \\
$x^k_{its}$        & \multicolumn{1}{l|}{Commitment status of generator $i$ in subperiod $s$ of period $t$ in scenario $k$.}               \\
$p^k_{its}$        & \multicolumn{1}{l|}{Power generation of generator $i$ in subperiod $s$ of period $t$ in scenario $k$.}                \\
$u^k_{its}$ & \multicolumn{1}{l|}{$1$ if generator $i$ starts up in subperiod $s$ of period $t$ in scenario $k$, and 0 otherwise.}  \\
$\nu^k_{its}$      & \multicolumn{1}{l|}{$1$ if generator $i$ shuts down in subperiod $s$ of period $t$ in scenario $k$, and 0 otherwise.} \\
$y^k_{ijts}$       & \multicolumn{1}{l|}{Switch status of transmission line $(i,j)$ in subperiod $s$ of period $t$ in scenario $k$.}                    \\
$f^k_{ijts}$       & \multicolumn{1}{l|}{Power flow along transmission line $(i,j)$ in subperiod $s$ of period $t$ in scenario $k$.}                    \\ \hline
\end{tabular}
\caption{Problem parameters and decision variables.}
\label{fig:problemParameters}
\end{table}
Next, we introduce the mathematical notation used in the formulation of the joint chance-constraint. 
This constraint aims to restrict the number of generators and lines that enter corrective maintenance with high probability. 
To this end, we let $\zeta_{ht}$ be a Bernoulli random variable which takes the value 1 if $t \ge \xi_h$, and $0$ otherwise, where $\xi_h$ represents the failure time of component $h$. 
Let us first define the following quantities $R_i(w) = \sum_{t \in \mathcal{\bar T}} \zeta_{it} w_{it}$ for every $i \in \mathcal{G}$ and $R_{ij}(z) = \sum_{t \in \mathcal{\bar T}} \zeta_{ijt} z_{ijt}$ for every $(i,j) \in \mathcal{L}$. Further, we let $E_{\mathcal{G}}(w)$ be the event that the total number of generators under corrective maintenance is less than a predetermined threshold $\rho_{\mathcal{G}}$ as:
\[E_{\mathcal{G}}(w) = \big \{ \sum_{i \in \mathcal{G}} R_i(w) \le \rho_{\mathcal{G}} \big \}.\] 

Here, a component is considered to enter corrective maintenance if its scheduled maintenance time is later than its time of failure. If the scheduled maintenance time is before the time of failure, then the maintenance is considered as predictive and prevents this undesirable failure event. We note that $R_i(w)$ can take at most the value 1, since the components can enter maintenance at most once within the planning horizon. Furthermore, this event is defined over the set $\mathcal{G}$ to capture the failure possibilities over all generators. 

Similarly, we let $E_{\mathcal{L}}(z)$ be the event that the total number of transmission lines under corrective maintenance is less than a predetermined threshold $\rho_{\mathcal{L}}$ as:
\[E_{\mathcal{L}}(z) = \big \{ \sum_{(i,j) \in \mathcal{L}} R_{ij}(z) \le \rho_{\mathcal{L}} \big \}.\] We define event $E_{\mathcal{H}}(v)$ as the intersection of events $E_{\mathcal{G}}(w)$ and $E_{\mathcal{L}}(z)$. 
%Notice that the quantities $R_i(w)$ and $R_{ij}(z)$ are defined for $i \in \mathcal{G''}$ and for $(i,j) \in \mathcal{L''}$, respectively. Before evaluating these quantities given any maintenance decision $v$,
Also, we let $R_i(w) = \zeta_{i|\mathcal{\bar T}|}$ for $i \in \mathcal{G''}$ and $R_{ij}(z) = \zeta_{ij|\mathcal{\bar T}|}$ for $(i,j) \in \mathcal{L''}$. Note that this is equivalent to the assumption that component $h \in \mathcal{H''}$ is not scheduled for maintenance within the planning horizon.

Now, we are ready to present the mathematical formulation of the joint chance-constrained stochastic optimization problem:
\begin{subequations} \label{eq:optimizationModel}
\begin{align}
    \min &\hspace{0.5em} \sum_{k \in \mathcal{K}} \pi^k\Big(\sum_{i \in \mathcal{G'}} \sum_{t=1}^{\xi^k_i-1}  C^{p}_{i} w_{it} + \sum_{i \in \mathcal{G'}}  \sum_{t =\xi^k_i : \xi^k_i  \neq |\mathcal{\bar T}|}^{|\mathcal{\bar T}|} C^{c}_i w_{it} \Big)\nonumber\\
        +&\hspace{.5em} \sum_{k \in \mathcal{K}} \pi^k\Big(\sum_{(i,j) \in \mathcal{L'}} \sum_{t=1}^{\xi^k_{ij}-1}  C^{p}_{ij} z_{ijt} + \sum_{(i,j) \in \mathcal{L'}} \sum_{t =\xi^k_{ij}:\xi^k_{ij} \neq |\mathcal{\bar T}| }^{|\mathcal{\bar T}|} C^{c}_{ij} z_{ijt} \Big)\nonumber\\
        +&\hspace{0.5em} \sum_{k \in \mathcal{K}} \sum_{i \in \mathcal{G}} \sum_{t \in  \mathcal{T}} \sum_{s \in \mathcal{S}}\pi^k(C_{i}^g p^k_{its} + C_i^n x^k_{its} + C_{i}^s u^k_{its})\nonumber\\
        +&\hspace{0.5em} \sum_{k \in \mathcal{K}} \sum_{i \in \mathcal{B}} \sum_{t \in  \mathcal{T}} \sum_{s \in \mathcal{S}}\pi^k C^{d}_i q^k_{its} \label{objective} \\
    \mathrm{s.t.}
        &\hspace{0.5em} \mathbb{P} (E_{\mathcal{G}}(w) \cap E_{\mathcal{L}}(z)) \ge 1 - \alpha \label{chanceConstr} \\
        &\hspace{0.5em} \sum_{t \in \mathcal{\bar 
        T}} w_{it} = 1 \quad i \in \mathcal{G'} \label{limitOnGenMaintenance} \\
        & \hspace{0.5em} \sum_{t \in \mathcal{\bar T}} z_{ijt} = 1 \quad (i,j) \in \mathcal{L'} \label{limitOnLineMaintenance}\\
        & \hspace{0.5em} x^k_{its} \le 1-\sum_{e=0}^{\tau^{p}_{{\scriptscriptstyle \mathcal{G}}}-1} w_{i(t-e)} \quad i \in \mathcal{G'}, s \in \mathcal{S}, t \in \{1,\dots, \xi^k_i+\tau^{p}_{{\scriptscriptstyle \mathcal{G}}}-1\}, k \in \mathcal{K} \label{logicalGenPMaintenance}\\
        & \hspace{0.5em} x^k_{its} \le \sum_{t'=1}^{\xi^k_i-1} w_{it'} \quad i \in \mathcal{G'}, s \in \mathcal{S}, t \in \{\xi^k_i,\dots, \xi^k_i+\tau^{c}_{{\scriptscriptstyle \mathcal{G}}}-1\}, k \in \mathcal{K} \label{logicalGenCMaintenance}\\ 
        & \hspace{0.5em} y^k_{ijts} \le 1-\sum_{e=0}^{\tau^{p}_{{\scriptscriptstyle \mathcal{L}}}-1} z_{ij(t-e)} \quad (i,j) \in \mathcal{L'}, s \in \mathcal{S}, t \in \{1,\dots, \xi^k_{ij}+\tau^{p}_{{\scriptscriptstyle \mathcal{L}}}-1\},  k \in \mathcal{K} \label{logicalLinePMaintenance}\\
        &\hspace{0.5em} y^k_{ijts} \le \sum_{t'=1}^{\xi^k_{ij}-1} z_{ijt'} \quad (i,j) \in \mathcal{L'}, s \in \mathcal{S}, t  \in \{\xi^k_{ij},\dots, \xi^k_{ij}+\tau^{c}_{{\scriptscriptstyle \mathcal{L}}}-1\}, k \in \mathcal{K} \label{logicalLineCMaintenance}\\
        &\hspace{0.5em} z_{ijt}+ y_{ijts}^k = 1  \quad (i,j) \in \mathcal{L'}, t \in \mathcal{T}, s \in \mathcal{S}, k \in \mathcal{K} \label{lineON}\\
        &\hspace{0.5em} \sum_{i' \in \mathcal{G}(i)} p^k_{i'ts} + q^k_{its} -d_{its} = \sum_{j \in \delta^+(i)} f^k_{ijts} - \sum_{j \in \delta^-(i)} f^k_{jits}  \quad i \in \mathcal{B}, t \in \mathcal{T}, s \in \mathcal{S}, k \in \mathcal{K}  \label{flowBalance}\\ 
        &\hspace{0.5em} B_{ij}(\delta^k_{its}-\delta^k_{jts}) = f^k_{ijts}  \quad (i,j) \in \mathcal{L}'', t \in \mathcal{T}, s \in \mathcal{S}, k \in \mathcal{K} \label{angle}\\
        & \hspace{0.5em} B_{ij}(\delta^k_{its}-\delta^k_{jts}) - M_{ij}(1 - y^k_{ijts}) \le f^k_{ijts} && \nonumber\\
        &\hspace{1.5em} \le B_{ij}(\delta^k_{its}-\delta^k_{jts}) + M_{ij}(1 - y^k_{ijts}) \quad (i,j) \in \mathcal{L'}, t \in \mathcal{T}, s \in \mathcal{S}, k \in \mathcal{K} \label{linearizedAngle}\\
        & \hspace{0.5em} -\Bar{f}_{ij} \le  f^k_{ijts}  \le \Bar{f}_{ij} \quad (i,j) \in \mathcal{L}'', t \in \mathcal{T},s \in \mathcal{S}, k \in \mathcal{K} \label{flowBounds}\\
        & \hspace{0.5em} -\Bar{f}_{ij}y^k_{ijts} \le  f^k_{ijts}  \le \Bar{f}_{ij}y^k_{ijts} \quad (i,j) \in \mathcal{L'}, t \in \mathcal{T},s \in \mathcal{S}, k \in \mathcal{K} \label{flowBoundsL}\\
        & \hspace{0.5em} p_i^{\min} x^k_{its}  \le p^k_{its} \le p_i^{\max} x^k_{its}  \quad i \in \mathcal{G}, t \in \mathcal{T}, s \in \mathcal{S}, k \in \mathcal{K} \label{powerBounds}\\
        & \hspace{0.5em} x^k_{it{\scriptscriptstyle(}s-1{\scriptscriptstyle)}} - x^k_{its} + u^k_{its} \ge 0 \quad i \in \mathcal{G}, t \in \mathcal{T}, s \in \mathcal{S}, k \in \mathcal{K} \label{startUp}\\
        & \hspace{0.5em} x^k_{its} - x^k_{it(s-1)} + \nu^k_{its} \ge 0 \quad i \in \mathcal{G}, t \in \mathcal{T}, s \in \mathcal{S}, k \in \mathcal{K} \label{shutDown}\\
        & \hspace{0.5em} -RD_i \le p^k_{its} -  p^k_{it{\scriptscriptstyle(}s-1{\scriptscriptstyle)}} \le RU_i \quad i \in \mathcal{G}, t \in \mathcal{T}, s \in \mathcal{S}, k \in \mathcal{K} \label{RampUpDown}\\
        & \hspace{0.5em} x^k_{its} - x^k_{it(s-1)} \le x^k_{its'} \quad i \in \mathcal{G}, t \in \mathcal{T}, s \in \mathcal{S}, s' \in \{s+1, s+MU_i-1\}, k \in \mathcal{K}  \label{minUp}\\
        & \hspace{0.5em} x^k_{it(s-1)} - x^k_{its} \le 1-x^k_{its'} \quad i \in \mathcal{G}, t \in \mathcal{T}, s \in \mathcal{S}, s' \in \{s+1, s+MD_i-1\}, k \in \mathcal{K}  \label{minDown}\\
        & \hspace{0.5em} w \in \{0,1\}^{|\mathcal{G'}| \times |\mathcal{\bar T}| }, \ z \in \{0,1\}^{|\mathcal{L'}| \times |\mathcal{\bar T}| }  \label{domainMaintenance}\\
        & \hspace{0.5em} x^k, v^k \in \{0,1\}^{|\mathcal{G}| \times |\mathcal{T}| \times |\mathcal{S}|}, y^k \in \{0,1\}^{|\mathcal{L'}| \times |\mathcal{T}| \times |\mathcal{S}|} \quad k\in \mathcal{K} \label{domainOperationalBinary}\\
        & \hspace{0.5em} u^k_{its} \in [0,1]\quad i \in \mathcal{G}, t \in \mathcal{T}, s \in \mathcal{S}, k \in \mathcal{K} \label{domainStartupContinuous}\\
        &\hspace{0.5em} \delta_{its}^k \in [\delta^{\min}_i, {\delta_i}^{\max}], q^k_{its} \ge 0 \quad i \in \mathcal{B}, t \in \mathcal{T}, s \in \mathcal{S}, k \in \mathcal{K}. \label{domainOperationalContinuous}
\end{align}
\end{subequations}

{The objective function \eqref{objective} aims to minimize the expected total cost, which consists of the expected maintenance costs of components and expected operational costs. For each component $h \in \mathcal{H'}$ under each scenario $k \in \mathcal{K}$, we incur its predictive maintenance cost if this component fails in that scenario, i.e., $\xi^k_h < |\mathcal{\bar T}|$, and a maintenance is scheduled before its failure time; or this component does not fail, i.e., $\xi^k_h = |\mathcal{\bar T}|$, and a maintenance is scheduled within the planning horizon. Otherwise, its corrective maintenance cost is incurred for the first case and no cost is incurred for the latter. The operational costs correspond to power generation, commitment, start-up, and demand curtailment. 

Constraint \eqref{chanceConstr} is a joint chance-constraint which holds with probability $1-\alpha$. This constraint limits the total number of generators and transmission lines going under corrective maintenance by predetermined thresholds $\rho_{\mathcal{G}}$ and $\rho_{\mathcal{L}}$, respectively.  Constraints \eqref{limitOnGenMaintenance} and \eqref{limitOnLineMaintenance} imply that exactly one maintenance must be scheduled within the extended planning horizon for every component. Constraints \eqref{logicalGenPMaintenance} and \eqref{logicalLinePMaintenance} ensure that if a component undergoes a predictive maintenance, it becomes unavailable until this predictive maintenance is completed whereas constraints \eqref{logicalGenCMaintenance} and \eqref{logicalLineCMaintenance} ensure the unavailability of a component from its failure time until a corrective maintenance is completed. On the other hand, constraint \eqref{lineON} guarantees that a transmission line is available unless it is under maintenance. Equation \eqref{flowBalance} represents the linearized power flow equations (Kirchhoff's Current Law) for each bus. Notice that heavily penalized power curtailment ($q_{its}$) is further added to \eqref{flowBalance}. This guarantees that we always obtain a feasible solution when the network fails to provide sufficient power supply to meet total power demand, which is a common practice in power systems. Equation \eqref{angle} is the power flow definition derived from Ohm's Law. When a transmission line is switched on $(y_{ijts}^k = 1)$, constraint \eqref{linearizedAngle} ensures that power flow is defined according to Ohm's Law, otherwise both upper bounds and lower bounds become redundant. Constraints \eqref{flowBounds} and \eqref{flowBoundsL} limit the power flow for each transmission line whereas constraint \eqref{powerBounds} limits the power generation for each generator. Constraints \eqref{startUp} and \eqref{shutDown} couple commitment status with start-up and shut-down variables, respectively. Constraint \eqref{RampUpDown} is the ramping constraint which guarantees that the power generation difference between consecutive hours does not exceed ramp-up and ramp-down limits. Constraints \eqref{minUp} and \eqref{minDown} are the minimum up and down times restrictions for each generator.} Constraints \eqref{domainMaintenance}, \eqref{domainOperationalBinary}, \eqref{domainStartupContinuous} and \eqref{domainOperationalContinuous} are for binary and nonnegativity restrictions. Note that binary start-up variables are relaxed to continuous variables since they are associated with positive cost coefficients in \eqref{objective}. Although this relaxation will expand the feasible region, it does not change the optimal value of our stochastic optimization problem (see, \citet{ONeillRelaxation}).

\subsection{Degradation Signal Modeling} \label{ss:DegradationSignalModelingandScenarioGeneration}
In this section, we explain our modeling framework for the RLDs of system components under a Bayesian setting. Figure \ref{fig:degrSignal} shows an example of degradation signal progress which has two main levels: Phase I and Phase II \citep{Nagi2006}. Phase I is referred to as the ``non-defective'' stage when a component does not show any sign of failure whereas Phase II is known as the ``defective'' stage in which degradation signal of system components aggressively deteriorates and results in failure when degradation signal reaches some predetermined threshold $\Lambda$. When modeling failure uncertainty of system components, we focus on the defective stage of their degradation signals.
% This file was created by tikzplotlib v0.9.8.
\begin{figure}[H]
\centering
\begin{tikzpicture}[scale=1]
\draw[gray, dashdotted, xstep=2cm, ystep=1.25cm] (0,0) grid (7,5.5);
\node at (0.8,5.25) {\small Phase I};
\node at (3.35 ,5.25) {\small Phase II};
\begin{axis}[
tick align=outside,
tick pos=left,
x grid style={white!69.0196078431373!black},
xlabel={\small Time},
xtick style={color=black},
y grid style={white!69.0196078431373!black},
ylabel={\small Signal amplitude},
ytick style={color=black}
]
\addplot [semithick, blue]
table {%
0 0
1 0.5
2 0
3 0.5
4 0
5 0.5
6 1
7 0.5
8 1.5
9 2
10 0
11 0.5
12 0
13 0.5
14 0
15 0.5
16 1
17 0.5
18 1.5
19 2
20 2
21 2.5
22 1.5
23 2.5
24 2
25 2.5
26 3
27 3.5
28 2.5
29 2
30 2
31 2.5
32 1.5
33 2.5
34 2
35 2.5
36 3
37 3.5
38 2.5
39 2
40 19.78249795
41 25.05951469
42 21.37352506
43 23.91028474
44 25.24197844
45 29.0459885
46 29.16189573
47 37.14910403
48 43.12502969
49 26.94568566
50 32.94775989
51 46.46328885
52 24.32785143
53 19.39564064
54 44.86262178
55 49.69923663
56 36.87459119
57 39.99465616
58 39.08648676
59 39.03561782
60 61.2863401
61 37.36292726
62 51.46584141
63 43.78654102
64 37.03374867
65 73.08469478
66 53.69020627
67 66.18947423
68 40.92200264
69 54.48554732
70 80.50571863
71 112.90786139
72 37.7643902
73 69.44604091
74 81.73381587
75 65.19713885
76 93.03320857
77 63.53540673
78 83.80691675
79 72.79074171
80 107.3188147
81 72.79917981
82 80.56049011
83 107.61557276
84 70.16394953
85 93.61520972
86 56.75853451
87 73.07611539
88 105.00403384
89 77.86691769
90 88.72940997
91 112.45555107
92 90.97179631
93 123.0826314
94 111.4758443
95 118.0247885
96 137.80104076
97 118.79524664
98 148.31992952
99 135.7699903
100 104.21742077
101 148.290765
102 125.83436735
103 111.81327973
104 129.50253659
105 151.33485755
106 137.13165331
107 162.40200748
108 156.14048531
109 143.9082656
110 106.72706123
111 129.54778744
112 137.52539882
113 155.16717508
114 155.36872796
115 144.93585962
116 145.28861024
117 162.71698328
};
\end{axis}
\end{tikzpicture}
\centering
\caption{An example degradation signal.}
\label{fig:degrSignal}
\end{figure}
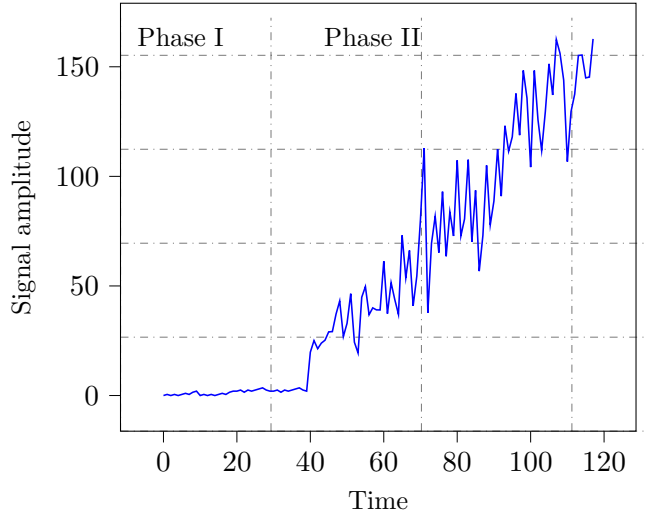
In this paper, we assume that we can identify degradation signal of each component using real-time sensor information. Then, we model each degradation signal as a stochastic continuous process and denote this process as $\mathcal{D} = \{D_h(t): t \ge 0 \}$ with $D_h: \mathbb{R} \rightarrow \mathbb{R}$ given by
\begin{align} \label{eq:degradationSignal}
    D_h(t) = \upsilon_h + \beta_h t + \sigma_h W(t),
\end{align} where $\upsilon_h$ is the initial signal amplitude and $\beta_h$ is the linear drift parameter for each $h \in \mathcal{H}$. The independent stochastic parameters $\upsilon_h$ and $\beta_h$ of the degradation signal model are presumed to follow some prior distributions which are assumed to be the same across every population (i.e., generators and transmission lines). The stochastic process $\mathcal{W} = \{W(t) : t \ge 0\}$ is the standard Brownian motion with $W(0) = 0$. We also assume that the standard deviation $\sigma_h$ of degradation signal of component $h$ is known and constant over the planning horizon. Furthermore, the standard deviation has the same value across every population. %As components degrade, their degradation signals tend to drastically increase. 
When the degradation signal level of a component exceeds the predefined threshold $\Lambda$, we assume that it fails. In particular, we define the failure time of component $h$ as the first passage time, i.e., $\xi_h = \min \{t \ge 0: D_h(t) \ge \Lambda\}$.

Next, we estimate the RLDs by using Bayesian inference combining both degradation signal model parameters and real-time condition-based sensor information. For every $h \in \mathcal{H}$, we assume that the prior distribution of the initial signal amplitude is $\upsilon_h \sim \mathcal{N} (\mu_0, \kappa^2_0)$ and the prior distribution of the linear drift is $\beta_h \sim \mathcal{N} (\mu_1, \kappa^2_1)$.  
%We also assume that the unknown parameters $\upsilon_h$ and $\beta_h$ follow the same family of prior distributions across each population.
We let $D_h(t_h^i)$ be the degradation signal level of component $h$ at time $t_h^i$. We define $D_h^i$ as the increment between times $t_h^i$ and $t_h^{i-1}$, given by $D_h^i = D_h(t_h^i) - D_h(t_h^{i-1})$ for $i =2, \dots, t^k_h$ with $D_h^1 = D_h(t_h^1)$ where $t^k_h, \ k \in \mathbb{Z_+}$ is the random observation time of the degradation signal of component $h$. Given the observed degradation signal data, we can mathematically derive the posterior distribution of the initial amplitude $\upsilon_h$ and the linear drift $\beta_h$ for every $h \in \mathcal{H}$ with a closed form expression (Proposition 2 by \citet{Nagi2005}). By using the posterior distribution of the drift parameter $\beta_h$, we estimate the RLD of component $h$ as in Proposition \ref{prop:inverseGauss}.
\begin{proposition} \label{prop:inverseGauss}
Given the observed signal increments $D_h^i$ at time $i =t_h^1, \dots ,t^k_h$ with prior parameters $(\upsilon_h, \beta_h)$, and the predefined failure threshold $\Lambda$, the posterior mean of the drift parameter of component $h$ is given by:
\begin{align} \label{posteriorDrift}
\mu_h' = \frac{(\kappa_1^2 \sum_{i=1}^{t^k_h} D_h^i + \mu_1 \sigma_h^2)(\kappa_0^2 + \sigma_h^2 {t}_h^1)- \kappa^2_1(D_h^1 \kappa^2_0 + \mu_0 \sigma_h^2 {t}_h^1) }{(\kappa_0^2 + \sigma^2_h {t}_h^1)(\kappa_1^2 {t}_h^k + \sigma_h^2) - \kappa_0^2 \kappa_1^2 {t}_h^1}.
\end{align}
Then, the remaining lifetime of component $h$ at time $t^k_h$ follows the inverse Gaussian distribution $\mathcal{IG}(t + t^k_h | \mu, \lambda)$ with shape parameter $\mu = \frac{\Lambda - \sum_{i=1}^{t^k_h} D_h^i}{\mu_h'}$ and scale parameter $\lambda = \frac{(\Lambda -\sum_{i=1}^{t^k_h} D_h^i )^2}{\sigma_h^2}$.
\end{proposition}

\subsection{Decomposition of the Stochastic Optimization Model} \label{ss:DecompositionStochOPT}
In this section, we present a compact formulation of the joint chance-constrained stochastic optimization model \eqref{eq:optimizationModel}. For ease of notation, we let $v$ be the decision vector containing the elements $v_{ht}$ corresponding to the maintenance decision of component $h \in \mathcal{H}'$ in maintenance period $t\in\mathcal{\bar T}$. Additionally, we define the binary vector $\eta^k$ consisting of commitment status, switch status and shut-down decisions under scenario $k$, and the continuous vector $\phi^k$ denoting the demand curtailment, voltage angle, power flow, power generation and start-up decisions under scenario $k$. We also let $E_{\mathcal{H}}(v)$ be the intersection of $E_{\mathcal{G}}(w)$ and $E_{\mathcal{L}}(z)$ (introduced in 
Section \ref{ss:DegradationSignalModelingandScenarioGeneration}). The compact formulation now can be stated as follows:
%
%
% Compact formulation
%
%
\begin{subequations}  \label{eq:compactModel}
\begin{align}
\min &\hspace{0.5em} \sum_{k \in \mathcal{K}} \pi^k \big( a^\top \eta^k + b^\top \phi^k + c^\top_{k} v)\\
\mathrm{s.t.}
&\hspace{0.5em} \mathbb{P}(E_{\mathcal{H}}(v)) \ge 1-\alpha \label{constr:compact_jointChance}\\
&\hspace{0.5em} A v= l \label{constr:compact_totalMaintenanceNumberMP}\\
&\hspace{0.5em} B^k v + D \eta^k \leq n    &k& \in \mathcal{K} \label{constr:compact_coupling} \\
&\hspace{3.60em} F \eta^k + G \phi^k \leq r  &k& \in \mathcal{K} \label{constr:compact_operational} \\
& \hspace{0.5em} v \in \{0,1\}^{|\mathcal{H}'| \times |\mathcal{\bar T}|} \label{constr:compact_sign_1}\\
&\hspace{0.5em} \eta^k \in \{0,1\}^{ (2|\mathcal{G}| + |\mathcal{L'}|) \times |\mathcal{T}| \times |\mathcal{S}|} &k& \in \mathcal{K}. \label{constr:compact_sign_2}
\end{align}
\end{subequations}
Next, we explain the correspondence between constraints in \eqref{eq:compactModel} and constraints in \eqref{eq:optimizationModel}. Constraint \eqref{constr:compact_jointChance} corresponds to the joint chance-constraint \eqref{chanceConstr} for which we propose two different representation in Section \ref{s:reformulationsJointChance}.
Constraint $\eqref{constr:compact_totalMaintenanceNumberMP}$ refers to the maintenance constraints \eqref{limitOnGenMaintenance} and \eqref{limitOnLineMaintenance} restricting the total number of maintenance schedules for each component within the planning horizon. Constraint \eqref{constr:compact_coupling} corresponds to the coupling constraints \eqref{logicalGenPMaintenance} - \eqref{lineON} between maintenance and operational decisions. Constraint \eqref{constr:compact_operational} represents the operational constraints \eqref{flowBalance} - \eqref{minDown} and domain restrictions \eqref{domainStartupContinuous} and \eqref{domainOperationalContinuous}. Constraints \eqref{constr:compact_sign_1} and \eqref{constr:compact_sign_2} correspond to the binary restrictions \eqref{domainMaintenance} and \eqref{domainOperationalBinary} for maintenance decision $v$ and operational decision $\eta$, respectively.

Next, we introduce the set of feasible maintenance decisions as $ \mathcal{\hat V} = \{ v \in \{0,1\}^{|\mathcal{H'}| \times |\mathcal{\bar T}|}: \eqref{constr:compact_jointChance}, \eqref{constr:compact_totalMaintenanceNumberMP}\}$. We can then reformulate \eqref{eq:compactModel} as a two-stage stochastic program given by:
\begin{align} \label{eq:twoStageCompact}
\min_{v} &\hspace{0.5em} \Big\{ \sum_{k \in \mathcal{K}} \pi^k \Big(c_k^{\top} v + \mathcal{Q}(v, \xi_k)\Big):  v \in \mathcal{\hat V} \Big\}, 
\end{align}
where $\mathcal{Q}(v,\xi_k)$ is the recourse function under scenario $k$ defined as follows:
\begin{align}  \label{eq:recourseFunction}
\mathcal{Q}(v,\xi_k) = \min_{\eta^k, \phi^k} \Big\{ &a^\top \eta^k + b^\top \phi^k : \ D \eta^k \leq n - B v, \ F \eta^k + G \phi^k \leq r, \ \eta^k \in \{0,1\}^{ (2|\mathcal{G}| + |\mathcal{L'}|) \times |\mathcal{T}| \times |\mathcal{S}|} \Big\}.
\end{align}
The first-stage decisions correspond to the maintenance decisions and the second-stage decisions correspond to the operational decisions. Note that the first-stage decisions are restricted to be binary and the second-stage decisions are restricted to be mixed-integer. We also denote the expected recourse function as $\mathcal{Q}(v, \xi)$ given by $\sum_{k \in \mathcal{K}} \pi^k \mathcal{Q}(v, \xi_k)$. 

Even with a small number of failure scenarios, the presented two-stage stochastic program \eqref{eq:twoStageCompact} can still be challenging to solve. Given a first-stage maintenance decision $v$ and a realization of random failure times $\xi_k$, we observe that maintenance periods under scenario $k$ become independent from each other. Fortunately, this allows us to further decompose each scenario subproblem into smaller and independent subproblems. We refer to this property as {\it time-decomposability} of scenario subproblems. We can formulate each smaller scenario subproblem under scenario $k$ by replacing \eqref{eq:recourseFunction} with $ \sum_{t \in \mathcal{T}}  \mathcal{Q}_t(v, \xi_k)$ where $\mathcal{Q}_t(v, \xi_k)$ is defined as follows:
\begin{align}  \label{eq:recourseFunctionDecomposed}
\mathcal{Q}_t(v,\xi_k) = \min_{\eta^k_t, \phi^k_t} \Big\{ &a^\top_t \eta^k_t + b^\top_t \phi^k_t : \ D_t \eta^k_t \leq n_t - B_t v, \ F_t \eta^k_t + G_t \phi^k_t \leq r_t, \ \eta^k \in \{0,1\}^{ (2|\mathcal{G}| + |\mathcal{L'}|) \times |\mathcal{S}|} \Big\}. 
\end{align}
Before moving to the next section, let us first obtain an equivalent MIP formulation for \eqref{eq:twoStageCompact} by introducing an auxiliary variable $\theta^k$ for the recourse function $\mathcal{Q}(v, \xi_k)$ for $k \in \mathcal{K}$. Consider the following mixed-integer master problem:
\begin{subequations} \label{eq:masterProblem}
\begin{align}
\min_{v,\theta} &\hspace{0.5em} \sum_{k \in \mathcal{K}} \pi^k (c_k^{\top}v + \theta^k) \label{constr:objMP}\\ 
\mathrm{s.t.}
&\hspace{0.5em} v \in \mathcal{\hat V} \label{constr:decompFeasibleRegion}\\
&\hspace{0.5em} \theta \ge L \label{constr:lowerBoundMP}\\
&\hspace{0.5em} (v, \theta) \in \Theta
\end{align}
\end{subequations}
As before, the decision vector $v$ corresponds to the maintenance decisions. The decision vector $\theta$ is used as an approximation of the recourse function. Constraints \eqref{constr:lowerBoundMP} are used to impose a lower bound on the recourse function. A trivial lower bound on the recourse function $\mathcal{Q}(v, \xi_k)$ is zero since all cost coefficients and their corresponding variables are nonnegative under scenario subproblem $k$. However, one can obtain a valid (and possibly better) lower bound $L^k$ on $\mathcal{Q}(v, \xi_k)$ by solving the following linear program:
\begin{align} \label{lowerBoundLP}
    L^k = \min_{\eta^k, \phi^k, v} \{a^\top \eta^k + b^\top \phi^k : \ D \eta^k  + B v \leq n , \ F \eta^k + G \phi^k \leq r, \eqref{constr:compact_totalMaintenanceNumberMP}, v \in [0,1]^{|\mathcal{H'}| \times |\mathcal{\bar T}|}\},
\end{align}
under scenario subproblem $k \in \mathcal{K}$. We refer to set $\Theta$ as the set of optimality cuts added to the relaxed master problem until some iteration. An optimality cut represents all possible values of the recourse function evaluated at different feasible solution. In particular, set $\Theta$ is called valid if for every $v \in \mathcal{\hat V}, (v, \theta) \in \Theta$ implies that $\theta^k \ge \mathcal{Q}(v, \xi_k)$ for $k \in \mathcal{K}$. Note that since \eqref{eq:twoStageCompact} has relatively complete recourse, we are not particularly interested in generating feasibility cuts which enforce the feasibility of each scenario subproblem. Suppose that  a valid and finite set of optimality cuts $\Theta$ indeed exists for the joint chance-constrained stochastic program \eqref{eq:twoStageCompact}, \eqref{eq:masterProblem} is then equivalent to \eqref{eq:twoStageCompact}. We can also obtain a different equivalent MIP formulation for \eqref{eq:twoStageCompact} by utilizing the time-decomposability of scenario subproblems. For this purpose, we replace $\theta^k$ in \eqref{constr:objMP} with $\sum_{t \in \mathcal{T}} \theta^k_t$. Similarly, one can obtain a valid lower bound $L^k_t$ on $\mathcal{Q}_t(v, \xi_k)$ by utilizing time-decomposability and solving \eqref{lowerBoundLP} under scenario $k$ in maintenance period $t$. 

We observe that set $\Theta$ may contain exponentially many constraints. Instead of adding all of these cuts to the problem, it might be more practical to consider a so-called relaxed master problem containing a small subset of $\Theta$ (possibly an empty set). In the next section, we propose an iterative algorithm with various algorithmic enhancements where the relaxed master problem is solved until we obtain the optimal solution to the two-stage joint chance-constrained stochastic program \eqref{eq:twoStageCompact}.

\section{Solution Methodology} \label{s:solutionMethodology}
In this section, we first explain our decomposition algorithm to solve \eqref{eq:twoStageCompact} and explain various algorithmic enhancements in detail (Section \ref{ss:decompositionAlgorithm}). Two different representations of the joint chance-constraint are explained in Section \ref{s:reformulationsJointChance}. The set of optimality cuts used in our proposed algorithm are presented in Section \ref{s:optimalityCutFamilies}. We further present a prepossessing step to address the potential redundancy in transmission flow limits in Section \ref{s:flowLimitAnalysis}. Finally, we use a SAA approach within the proposed decomposition algorithm to obtain statistical bounds on the true optimality gap in Section \ref{s:sampleAverageApproximation}.

\subsection{Decomposition Algorithm} \label{ss:decompositionAlgorithm}

We benefit from the features of the integer L-shaped method to develop a decomposition-based algorithm (Algorithm \ref{alg:decomposition}) to solve our two-stage joint chance-constrained stochastic program \eqref{eq:twoStageCompact}. Given a first-stage decision, solving many similar mixed-integer scenario subproblems can be computationally expensive. We propose an algorithmic enhancement to avoid this situation by exploiting the time decomposability of scenario subproblems and identifying the ``status'' of system components.

Let us first devise a concept of status of system components, which is utilized in our decomposition algorithm. This concept is used to characterize the availability of system components. Given a feasible maintenance decision $v \in \mathcal{\hat V}$, the status of component $h \in \mathcal{H'}$, denoted by $u_{ht}^k(v)$, takes a value of $1$ if the corresponding component is available in maintenance period $t \in \mathcal{T}$ under scenario $k \in \mathcal{K}$, and $0$ otherwise. We corroborate this concept with a simplified instructive example. Consider our stochastic optimization problem \eqref{eq:optimizationModel} under a single scenario with $|\mathcal{G'}| = |\mathcal{L'}| = 1$, $|\mathcal{T}| = 4$ and $(\tau^{p}_{{\scriptscriptstyle \mathcal{G}}}, \tau^{c}_{{\scriptscriptstyle \mathcal{G}}},\tau^{p}_{{\scriptscriptstyle \mathcal{L}}}, \tau^{c}_{{\scriptscriptstyle \mathcal{L}}}, \xi_1^1, \xi_2^1) = (1,2,1,2,1,4)$. Suppose we are given a feasible maintenance decision $v = [0,1,0,0,0;0,1,0,0,0]$, which corresponds to the case where the components enter maintenance in the second period and let these components correspond to the  generator and the transmission line, respectively. The status vector of component $1$ is $[0,0,1,1]$ as the scheduled predictive maintenance at period 2 is at a later period than its failure time at period 1. Thus, it is under corrective maintenance and is unavailable for two consecutive maintenance periods. On the other hand, the status vector of component $2$ is $ [1,0,1,1]$ since the scheduled predictive maintenance at period 2 prevents the failure at period 4, and this component is only unavailable for one maintenance period. By combining these status vectors of components, we obtain  $u^1(v) = [0,0,1,1;1,0,1,1]$ where column $t$ consists of the status of components in maintenance period $t$ for every $t \in \mathcal{T}$. Now, suppose we are given another feasible maintenance decision under the same single scenario problem as $ \tilde v = [0,0,1,0,0;0,0,0,1,0]$. Similarly, we obtain $u^1(\tilde v) = [0,0,1,1;1,1,0,1]$. In the remainder of our paper, we let $u^k_t(v)$ denote the column $t$ of $u^k(v)$ given maintenance decision $v$ for every $k \in \mathcal{K}$ and $t \in \mathcal{T}$. Observe that $u^1_1(v)$ and $u^1_4( v)$ are the same with $u^1_1(\tilde v)$ and $u^1_4(\tilde v)$, respectively. This implies that the components have the same status for the first and fourth periods under this scenario for the solutions $v$ and $\tilde{v}$. To generalize, decomposing a scenario subproblem into maintenance periods given different feasible maintenance decisions may yield to some identical scenario subproblems depending on the availability of the components. Eventually, we exploit this observation and adapt the concept of status in our decomposition algorithm. This allows us to uniquely determine the nature of each scenario subproblem given different feasible maintenance (first-stage) decisions. 

Next, we explain our proposed decomposition-based algorithm to solve \eqref{eq:twoStageCompact}, which is summarized in Algorithm \ref{alg:decomposition}. For representing the joint chance-constraint \eqref{constr:compact_jointChance}, this algorithm considers both an exact reformulation and a deterministic safe approximation, which are further explained in Section \ref{s:reformulationsJointChance} in detail. When the exact reformulation is used, Algorithm \ref{alg:decomposition} employs a cutting-plane method over the joint chance-constraint \eqref{constr:compact_jointChance} as follows: At the beginning of each iteration, we obtain a maintenance decision $v$ by solving the relaxed master problem \eqref{eq:masterProblemCover}. We call the separation subroutine \texttt{RepresentChance$(v)$} with input $v$ to check whether this maintenance decision is feasible with respect to \eqref{constr:compact_jointChance}. When the infeasibility of $v$ is detected, a cutting plane is generated and added to set $\mathcal{C}$. This separation subroutine and violated cover inequalities are further explained in Section \ref{ss:exactReformulation}. When the deterministic safe approximation is used, we always obtain a feasible maintenance decision $v \in \mathcal{\hat V}$ since this approximation provides a conservative representation of the joint chance-constraint. After obtaining a feasible maintenance decision $v \in \mathcal{\hat V}$ and failure uncertainty is revealed for every scenario, Algorithm \ref{alg:decomposition} proceeds to the second-stage. For storing the status vectors in period $t$, we define set $\Upsilon_t$ which corresponds to the set of unique status vectors identified at that iteration. Also, we define set $\Psi_t$ to represent the set of all unique status vectors until termination within Algorithm \ref{alg:decomposition}. After identifying the unique status vectors (Step \ref{alg:decompStatus}) for every maintenance period $t \in \mathcal{T}$, Algorithm \ref{alg:decomposition} continues to solve only the subproblems with these newly identified status vectors (Step \ref{alg:decompStatusParallel}). By restricting ourselves to set $\Psi_t$, it suffices to solve $ \sum_{t \in \mathcal{T}} |\Psi_t|$ many scenario subproblems until termination, which could be significantly less than the total number of scenario subproblems to be solved throughout the algorithm. Finally, Algorithm \ref{alg:decomposition} initializes \texttt{OptimalityCut($v, \xi, L$)} with input $v, \xi$ and $L$ to generate and add optimality cuts to set $\Theta$ in Step \ref{alg:decompOptimalityCuts}. Algorithm \ref{alg:decomposition} continues to iterate until a relative optimality gap within a tolerance $\epsilon$ is achieved. The implementation of our decomposition-based algorithm is in parallel in order to achieve computational efficiency. At the initialization of Algorithm \ref{alg:decomposition} (Step \ref{alg:decompInitial}), linear relaxations of the subproblems are solved to obtain lower bounds on the objectives of these problems. Due to the independence of scenario subproblems and time-decomposability, these subproblems are solved within a distributed environment. Similarly, scenario subproblems corresponding to the unique status vectors (Step \ref{alg:decompStatusParallel}) are solved to optimality with parallelization. 
%
%
% ALGORITHM DECOMPOSITION
%
%
%
\begin{algorithm}
\caption{Decomposition}
\label{alg:decomposition}
\begin{algorithmic}[1]
\REQUIRE $A, B, D, F, G, a,b,c, l, n, r, \epsilon, \mathcal{Q} : (v, \xi) \rightarrow \mathbb{R}$.
\ENSURE $\epsilon$-optimal solution $v^*$ and $\epsilon$-optimal objective value $c^*$.
\STATE Set $UB = \infty, LB = - \infty, \Theta = \mathcal{C} = \emptyset, \Psi_t = \emptyset$ for all $t \in \mathcal{T}$.
\STATE Compute the lower bound $L$ of $\mathcal{Q}(v, \xi)$ \text{\it (in parallel)}. \label{alg:decompInitial}
\WHILE{$LB / UB < 1 - \epsilon$}
\IF{{the joint chance-constraint representation is exact}} 
\STATE Solve a relaxed master problem \eqref{eq:masterProblemCover} to obtain a solution $(v,\theta)$. \label{alg:decompExactRoutine}
\STATE $flagFeasible \gets \texttt{RepresentChance($v$)}$. \label{alg:decomp_chanceExact}
\ELSE 
\STATE  Solve a relaxed master problem  \eqref{eq:masterProblem} to obtain a feasible solution $(v,\theta)$.
\STATE $flagFeasible \gets {\bf true}$.
\ENDIF
\IF{$flagFeasible$ is {\bf true}}
\STATE ${LB} \gets  max (\sum_{k \in \mathcal{K}} \pi^k (c_k^{\top}v + \sum_{t \in \mathcal{T}} \theta^k_t), \ LB)$.
\STATE Identify the status vector $u_t^k(v) \in \{0,1\}^{|\mathcal{H}'|}$ for $(k, t) \in \mathcal{K} \times \mathcal{T}$. \label{alg:decompStatus}
\STATE Set $\Upsilon_t = \Gamma_t =\emptyset$ for all $t \in \mathcal{T}$.
%%%
%\STATE Identify the unique status vectors in period $t$ and add them to set $\Upsilon_t$ for all $t \in \mathcal{T}$.
%\STATE Identify the scenario subproblem tuples which have the same status vectors in period $t$ and add them to set $\Gamma_t$ for all $t \in \mathcal{T}$.
\FORALL{$(k,t) \in \mathcal{K} \times \mathcal{T}$}
\IF{$u_t^k(v) \in \Psi_t$}
\STATE Find an index $\hat k$ such that $u_t^k(v) = u_t^{\hat{k}}(v) \in \Psi_t$, and $ \Gamma_t \gets \Gamma_t \bigcup \{( \hat{k}, k) \}$. %\{Store the index of subproblem $k$ in period $t$ associated with the status of subproblem $\hat k$ in period $t$.\}.
%\{Store scenario subproblem indexes with the same status.\}.
\ELSE 
\STATE $\Upsilon_t \gets  \Upsilon_t \bigcup \{u_t^k(v)\}$.% \{Add the unique status identified to set $\Upsilon_t$.\}.
\ENDIF
\ENDFOR
\STATE $\Hat{\Upsilon} \gets \bigcup_{t \in \mathcal{T}} \Upsilon_t$, and $\Psi_t \gets \Psi_t \bigcup \Upsilon_t $ for all $t \in \mathcal{T}$.
%%%
\FORALL{$ u_t^k(v) \in \Hat{\Upsilon}$ \text{\it (in parallel)}} \label{alg:decompStatusParallel}
\STATE Solve scenario subproblem \eqref{eq:recourseFunctionDecomposed} associated with $u_t^k(v)$ and obtain $\mathcal{Q}_t(v,\xi_k)$.
%\ENDFOR
\ENDFOR
\FORALL{$t \in \mathcal{T}$}
\STATE $\mathcal{Q}_t(v, \xi_{\hat{k}}) \gets \mathcal{Q}_t(v, \xi_k)$ for all $( \hat{k}, k) \in \Gamma_t$. %\{Map the recourse values of the unsolved subproblems in set $\Gamma_t$.\}.
\ENDFOR
\STATE $c^* \gets \sum_{k \in \mathcal{K}} \pi^k (c_k^{\top} v + \sum_{t \in \mathcal{T}} \mathcal{Q}_t(v, \xi_k))$.
\IF{$UB > c^*$}
\STATE $(UB, v^*) \gets (c^*, v)$.
\ENDIF
\STATE Initialize \texttt{OptimalityCut($v, \xi, L$)} and add the optimality cut to set $\Theta$. \label{alg:decompOptimalityCuts}
\ELSE
\STATE {\bf continue}
%\STATE Initialize \texttt{OptimalityCut($v, \xi, L$)} and add the cover inequality to set $\mathcal{C}$.
\ENDIF
\ENDWHILE
\RETURN Optimal solution $v^*$ and optimal value $c^*$.
\end{algorithmic}
\end{algorithm}

%
%
%
% chance-constraint REPRESENTATION 
%
%
%
\subsection{Reformulations of the Joint Chance-Constraint} \label{s:reformulationsJointChance}
In this section, we specify two different approximations of the joint chance-constraint \eqref{constr:compact_jointChance}. First recall the set of feasible maintenance decisions defined as $\mathcal{\hat V} = \{ v \in \{0,1\}^{|\mathcal{H'}| \times |\mathcal{\bar T}|}: \eqref{constr:compact_jointChance},  \eqref{constr:compact_totalMaintenanceNumberMP}\}$. We first obtain an approximation of $\mathcal{\hat V}$ by using a probability oracle which provides the exact value of the left-hand side of the joint chance-constraint \eqref{constr:compact_jointChance} and further, prove that this approximation is in fact exact. However, this exact representation requires an exponential reformulation of $\mathcal{\hat V}$. Thus, we employ a cutting-plane method using the separation subroutine \texttt{RepresentChance($v$)} within Algorithm \ref{alg:decomposition} to efficiently solve our two-stage joint chance-constrained stochastic program \eqref{eq:twoStageCompact}. Still, this cutting-plane method may not be efficient as the number of components increases. To address this issue, we also develop a second-order cone programming (SOCP) based deterministic safe approximation of \eqref{constr:compact_jointChance}. 

\subsubsection{Exact Reformulation} \label{ss:exactReformulation} Suppose we are given a feasible maintenance decision $v = (w,z) \in \mathcal{\hat V}$. Let us first consider the quantity $R_i(w) = \sum_{t \in \mathcal{\bar T}} \zeta_{it} w_{it}$ for every $i \in \mathcal{G}$ as defined in Section \ref{ss:StochasticOptimizationModel}. Here, we calculate the failure probability for generators $i \in \mathcal{G}$ within the planning horizon. Recall constraint \eqref{limitOnGenMaintenance} in the joint chance-constrained stochastic optimization model \eqref{eq:optimizationModel}, that is, $\sum_{t \in \mathcal{\bar T}} w_{it} = 1$ for $i \in \mathcal{G'}$. This implies that the quantity $R_i(w)$ is a Bernoulli random variable with the success probability $\mathbb{P}(\xi_i \le m_i(w))$ where $ m_i(w)$ is the maintenance period in which a maintenance is scheduled for generator $i \in \mathcal{G'}$, and if no maintenance is scheduled, we let $m_i (w)$ be $|\mathcal{T}|$ for generator $i \in \mathcal{G'}$. Recall also the following assumption that there is no maintenance scheduled for generators $i \in \mathcal{G''}$, thus, we also let $m_i (w)$ be $|\mathcal{T}|$ for $i \in \mathcal{G''}$. Then, the quantity $R_i(w)$ is also a Bernoulli random variable with the success probability $\mathbb{P}(\xi_i \le |\mathcal{T}|)$ for $i \in \mathcal{G''}$. 

Let us now consider the quantity $R_{ij}(z) = \sum_{t \in \mathcal{\bar T}} \zeta_{ijt} z_{ijt}$ for every $(i,j) \in \mathcal{L}$ as defined in Section \ref{ss:StochasticOptimizationModel}. The similar results also hold for the transmission lines with constraint \eqref{limitOnLineMaintenance} in the stochastic optimization model. Therefore, the quantity $\xi_{ij}(z)$ is a Bernoulli random variable with the success probability $\mathbb{P}(\xi_{ij} \le  m_{ij}(z))$ for $(i,j) \in \mathcal{L'}$. Similarly, the quantity $R_{ij}(z)$ is also a Bernoulli random variable with the success probability $\mathbb{P}(\xi_{ij} \le |\mathcal{T}|)$ for $(i,j) \in \mathcal{L''}$.

Next, we define the following random variables $\hat \zeta_{\mathcal{G}}(w) = \sum_{i \in \mathcal{G}} R_i(w)$ and $\hat \zeta_{\mathcal{L}}(z) = \sum_{(i,j) \in \mathcal{L}} R_{ij}(z)$ as the sum of independent Bernoulli random variables.
%\begin{remark} \label{remark:poisson_gen}
%Given a feasible maintenance decision $w$ for generators, the random variable $\hat \zeta_{\mathcal{G}}(w)$ follows a Poisson Binomial distribution with success probabilities $\{\mathbb{P}(R_i(w) \le \hat \rho_i(w)); \ i \in \mathcal{G}\}$.
%\end{remark}
\begin{remark} \label{remark:poissonComp}
The random variables $\hat \zeta_{\mathcal{G}}(w)$ and $\hat \zeta_{\mathcal{L}}(z)$ follow Poisson Binomial distributions with success probabilities $\{\mathbb{P}(\xi_i \le m_i(w)); \ i \in \mathcal{G}\}$ and $\{\mathbb{P}(\xi_{ij} \le m_{ij}(z)); \ (i,j) \in \mathcal{L}\}$, respectively.
\end{remark}
By Remark \ref{remark:poissonComp}, we observe that the left hand-side expression of the joint chance-constraint \eqref{constr:compact_jointChance} is equivalent to the following joint cumulative distribution function of two Poisson Binomial random variables, i.e., $\mathbb{P}(\hat \zeta_{\mathcal{G}}(w) \le \rho_{\mathcal{G}}, \hat \zeta_{\mathcal{L}}(z) \le \rho_{\mathcal{L}})$. By using the independence of these two random variables, joint chance-constraint \eqref{constr:compact_jointChance} can be rewritten as:
\begin{align}
    \mathbb{P}(\hat \zeta_{\mathcal{G}}(w) \le \rho_{\mathcal{G}}) \mathbb{P}( \hat \zeta_{\mathcal{L}}(z) \le \rho_{\mathcal{L}}) \ge 1-\alpha. \label{eq:chanceBernoulli}
\end{align}
Finally, we can recast $\mathcal{\hat V}$ as follows:
\begin{align}
    \mathcal{\hat V} = \{(w,z) \in \{0,1\}^{|\mathcal{G'}| \times |\mathcal{L'}|}: \eqref{eq:chanceBernoulli}, \eqref{constr:compact_totalMaintenanceNumberMP}\}. \label{eq:approxPoissonBinomial}
\end{align}
Next, we derive the exact representation of our joint chance-constraint \eqref{eq:chanceBernoulli}. For any maintenance decision $v = (w,z)$, we first suppose that there exists a probability oracle $\mathcal{P}(v)$ which recognizes $v$ and provides the exact value of the left hand-side of relation \eqref{eq:chanceBernoulli}. We then introduce the index set $N = \{ (h,t): h \in \mathcal{H}', t \in \mathcal{\bar T}\}$. We call a set $C \subseteq N$ a {\it scheduling set} if it satisfies the following property: 
\[\text{there exists a unique } t(h) \in \mathcal{\bar T} : (h, t(h)) \in C, \ \text{for every } h \in \mathcal{H'}.\] 
In other words, this set includes a unique maintenance period for every component from set $\mathcal{H}'$. Furthermore, we call a scheduling set $C \subseteq N$ a {\it cover} for $\mathcal{\hat V}$ if $\mathcal{P}(v(C)) < 1-\alpha$, that is, $v(C) \notin \mathcal{\hat V}$.

\begin{proposition} \label{prop:feasibleCover}
  Given a cover $C \subseteq N$, the following set of cover inequalities is valid for $\mathcal{\hat V}$:
\begin{align}
   \sum_{(h,t) \in C} v_{ht} \le |\mathcal{H}'| - 1. \label{cut:coverIneq}
\end{align}
\end{proposition}
\begin{proof}[Proof.]
Let $\hat v \in \mathcal{\hat V}$. Assume for the sake of contradiction that $\sum_{(h,t) \in C} \hat v_{ht} \ge |\mathcal{H}'|$. Since $\sum_{t \in \mathcal{\bar T}} \hat v_{ht} = 1$ for every $h \in \mathcal{H}'$, it must be that $\hat v_{ht} = 1$ for every $(h,t) \in C$. Then contradiction follows immediately since $\mathcal{P}(\hat v) < 1-\alpha$ implies that $\hat v \notin \mathcal{\hat V}.$
\Halmos
\end{proof}
We can ensure that at least one of the elements of the maintenance schedule defined by the set $C$ need to be rescheduled by simply using a cover inequality in \eqref{cut:coverIneq}. In particular, if we find every valid cover inequality defined in \eqref{cut:coverIneq} for every cover which is a subset of $N$, we can obtain an equivalent formulation of $\mathcal{\hat V}$. We show this result in Proposition \ref{prop:cover_equivalence}.
\begin{proposition} \label{prop:cover_equivalence}
Consider the following set:
\begin{align} 
    \mathcal{\hat V}_1 = \{ (w,z) \in \{0,1\}^{|\mathcal{G'}| \times |\mathcal{L'}|}: \sum_{(h,t) \in C} v_{ht} \le |\mathcal{H}'| - 1, \eqref{constr:compact_totalMaintenanceNumberMP}, \forall C \subseteq N \text{ s.t. $C$ is a cover} \}. \label{eq:coverIneqSet}
\end{align}
Sets defined in \eqref{eq:coverIneqSet} and \eqref{eq:approxPoissonBinomial} are equivalent.
\end{proposition}
\begin{proof}[Proof.]
We only prove that $\mathcal{\hat V}_1 \subseteq \mathcal{\hat V}$ by contraposition since the converse is proven in Proposition \ref{prop:feasibleCover}. Let $\tilde v = (\tilde w, \tilde z) \notin \mathcal{\hat V}$ such that constraints \eqref{constr:compact_totalMaintenanceNumberMP},\eqref{constr:compact_sign_1} hold for $\tilde v$ but $\mathbb{P}(\hat \zeta_{\mathcal{G}}(\tilde w) \le \rho_{\mathcal{G}}, \ \hat \zeta_{\mathcal{L}}(\tilde z) \le \rho_{\mathcal{L}}) < 1-\alpha$. Then, there must exist at least one subset $\tilde C \subseteq N$ such that $\tilde C$ is a cover, which directly implies by definition that $\tilde v \notin \mathcal{\hat V}_1.$
\Halmos
\end{proof}

Unfortunately, set \eqref{eq:coverIneqSet} requires finding an exponential number of valid inequalities to obtain the feasible space for $\mathcal{\hat V}$. In the remainder of this section, we address a separation subroutine using the probability oracle $\mathcal{P}$ to check the feasibility status of the current maintenance solution and further find valid cover inequalities, if such equalities exist (see, for instance, \citet{Kucukyavuz2019}). We first consider the relaxed master problem of \eqref{eq:masterProblem}:
\begin{align}  \label{eq:masterProblemCover}
    \min \{ \sum_{k \in \mathcal{K}} \pi^k (c_k^{\top}v + \theta^k) : \ \theta \ge L, \ (v,\theta) \in \Theta, \ v \in \mathcal{\hat V} \cap \mathcal{C} \},
\end{align}
where $\mathcal{C}$ is the set of cover inequalities generated and added to the relaxed master problem until some iteration. Algorithm \ref{alg:decomposition} starts with a possibly  empty  subset of $\mathcal{ C}$. After obtaining a maintenance decision $\hat v = (\hat w,\hat z)$ at the end of step \ref{alg:decompExactRoutine}, we initialize the separation subroutine (Algorithm \ref{alg:representChance}) which employs the probability oracle $\mathcal{P}(\hat v)$ to compute the exact value of the left hand-side of \eqref{constr:compact_jointChance} when  $\hat v$ is a feasible maintenance decision. Algorithm \ref{alg:decomposition} leaves the separation subroutine without generating a valid inequality if the current solution is feasible. Otherwise, we define a cover as $C= \bigcup_{h \in \mathcal{H'}} \{(h,t) \in N: \hat v_{ht} = 1, \ t \in \mathcal{\bar T}\}$ and generate a cover inequality as in \eqref{cut:coverIneq} to separate the current solution from the set of maintenance decisions and add the corresponding inequality to set $\mathcal{C}$. 

Next, we explain the monotonicity property of the probability oracle in our paper. Given any maintenance decision $v$, let us define the index set of components and maintenance periods as follows: \[\mathcal{I}(v) = \bigcup_{h \in \mathcal{H'}} \{ (h,t) : v_{ht} = 1, \ t \in \mathcal{\bar T} \}.\]
We say that $\mathcal{P}$ is monotonically non-increasing if any $v', v''$ pair has the following property:  
\[(h, t') \le (h, t'') \text{ for } (h,t') \in \mathcal{I}(v'), \ (h,t'') \in \mathcal{I}(v'') \text{ and } h \in \mathcal{H'}.\]
This implies that $ \mathcal{P}(v') \ge \mathcal{P}(v'')$, i.e., $\mathbb{P}(\hat \zeta_{\mathcal{G}}(w') \le \rho_{\mathcal{G}}) \mathbb{P}(\hat \zeta_{\mathcal{L}}(z') \le \rho_{\mathcal{L}}) \ge \mathbb{P}(\hat \zeta_{\mathcal{G}}(w'') \le \rho_{\mathcal{G}}) \mathbb{P}(\hat \zeta_{\mathcal{L}}(z'') \le \rho_{\mathcal{L}})$. We state this property in Proposition \ref{prop:monotonicity}.

\begin{proposition} \label{prop:monotonicity}
Probability oracle $\mathcal{P}(v)$ is a monotonically non-increasing function in $v$. 
\end{proposition}
The proof of Proposition \ref{prop:monotonicity} is given in Appendix \ref{app:monotonicity} which leverages the fact that the random variables have Poisson Binomial distribution as discussed in Remark \ref{remark:poissonComp}. We use Proposition \ref{prop:monotonicity} to strengthen the formulation in \eqref{cut:coverIneq}. 
Without loss of generality, we assume a pair of maintenance decisions $v' \neq v''$ with the following property:
\begin{align*}
    \text{there exists a unique } h_* \text{ and } t'_* < t''_* : \ &(h_*, t'_* ) \in \mathcal{I}(v'), \ (h_*, t''_* ) \in \mathcal{I}(v''), \ (h', t') = (h'', t'') \\ &\text{ for } (h', t') \in \mathcal{I}(v') \setminus \{(h_*, t'_*)\}, \ (h'', t'') \in \mathcal{I}(v'') \setminus \{(h_*, t''_*)\}.
\end{align*}
In other words, component $h_*$ is scheduled for maintenance in period $t'_*$ under decision  $v'$, and it is scheduled for maintenance in a later period $t''_*$ than $t'_*$ under decision $v''$. For each component $h \in \mathcal{H'} \setminus \{h_*\}$, maintenance schedules are the same under both decisions. By Proposition \ref{prop:monotonicity}, we have $\mathcal{P}(v') > \mathcal{P}(v'')$. We observe that if $v'$ is infeasible w.r.t. constraint \eqref{constr:compact_jointChance}, i.e., $\mathcal{P}(v') < 1-\alpha$, then clearly $v''$ is also infeasible w.r.t. constraint \eqref{constr:compact_jointChance}. In particular, we observe that any other maintenance plan for component $h_*$ in a later period than $t'_*$ will lead to infeasibility. We can extend this observation when multiple components have different maintenance schedules under decisions $v'$ and $v''$. Then, we can strengthen \eqref{cut:coverIneq} as follows: when an infeasible maintenance decision $v$ is obtained within Algorithm \ref{alg:decomposition} at some iteration, we can generate a cover $C$ as explained previously. Bearing in mind our observation, we define a set $E(C)$ depending on $C$ as follows:
\begin{align}
    E(C)= \bigcup_{h \in \mathcal{H'}} \{(h, t): t = t(h), t(h)+1, \dots, |\mathcal{\bar T}| \text{ where } (h,t(h)) \in C \} \label{eq:extendedCover}.
\end{align}
We call a set $E(C)$ defined as in \eqref{eq:extendedCover} an {\it extended cover} for $\mathcal{\hat V}$ if $C$ is a cover. By using this set, we can obtain stronger cover inequalities than \eqref{cut:coverIneq}. We state our claim in Proposition \ref{prop:ExtCover}.
\begin{proposition} \label{prop:ExtCover}
Given a cover $C \subseteq N$, the following set of extended cover inequalities is valid and stronger than the set of cover inequalities given by \eqref{cut:coverIneq} whenever $E(C) \setminus C \neq \emptyset$:
\begin{align}
   \sum_{(h,t) \in E(C)}   v_{ht} \le |\mathcal{H}'| - 1. \label{eq:extendedCoverIneq}
\end{align}
\end{proposition}
\begin{proof}[Proof.]
By using Proposition \ref{prop:feasibleCover} and Proposition \ref{prop:monotonicity}, we have the proof of validity. To prove the strength of \eqref{eq:extendedCoverIneq},
suppose $E(C) \setminus C \neq \emptyset$ holds, then we have $\sum_{(h,t) \in E(C)}   v_{ht} \ge \sum_{(h,t) \in C}  v_{ht}$ since $C \subsetneq E(C).$ This completes the proof.
\Halmos
\end{proof}
Finally, we present our separation subroutine (Algorithm \ref{alg:representChance}) within Algorithm \ref{alg:decomposition}.
%
%
% ALGORITHM REPRESENTCHANCE
%
%
%
\begin{algorithm}[H]
\caption{\texttt{RepresentChance}}
\label{alg:representChance}
\begin{algorithmic}[1]
\REQUIRE $v, \mathcal{P}: v \rightarrow \mathbb{R}$.
\ENSURE {\bf true} if $v$ is feasible w.r.t. \eqref{constr:compact_jointChance}, {\bf false} otherwise.
%\STATE Compute $\mathbb{P}(\hat \zeta_{\mathcal{G}}(w) \le \rho_{\mathcal{G}})$ and $ \mathbb{P}(\hat \zeta_{\mathcal{L}}(z) \le \rho_{\mathcal{L}})$.
\STATE Compute $\mathcal{P}(v)$.
\IF{$\mathcal{P}(v) \ge 1-\alpha$}
\RETURN {\bf true}
\ELSE
\STATE Generate and add the cover inequality of form \eqref{eq:extendedCoverIneq} to set $\mathcal{C}$.
\RETURN {\bf false}
\ENDIF
\end{algorithmic}
\end{algorithm}

\subsubsection{Deterministic Safe Approximation} \label{ss:deterministicSafeApproximation}
As an alternative representation of the joint chance-constraint \eqref{constr:compact_jointChance}, we propose an SOCP-based safe approximation. The proposed safe approximation is an extension of the deterministic safe approximation of a single chance-constraint by \citet{Basciftci2018} by introducing two additional continuous variables and reformulating $\mathcal{\hat V}$ as a second-order conic set by lifting it to a higher-dimensional space.
\begin{proposition}
The following system of equations provides a safe approximation of the joint chance-constrained set $\mathcal{\hat V}$, i.e., any maintenance decision $v$ satisfying \eqref{SOCP_Chance} and  \eqref{constr:compact_totalMaintenanceNumberMP} belongs to set $\mathcal{\hat V}$:
\begin{subequations} \label{SOCP_Chance}
\begin{align} 
&\sum_{i \in \mathcal{G}} \sum_{t \in \mathcal{\bar T}} \mathbb{E}[\zeta_{it}] w_{it} \leq \rho_{\mathcal{G}}(1-\bar \alpha_{\mathcal{G}})  \label{eq:safeLinearGen}\\
&\sum_{(i,j) \in \mathcal{L}} \sum_{t \in \mathcal{\bar T}} \mathbb{E}[\zeta_{ijt}]z_{ijt} \leq \rho_{\mathcal{L}}(1- \bar \alpha_{\mathcal{L}}) \label{eq:safeLinearLine}\\
&\bar \alpha_{\mathcal{G}} \bar \alpha_{\mathcal{L}} \ge 1 - \alpha \label{eq:safeSOCP}\\ &\bar \alpha_{\mathcal{G}}, \bar \alpha_{\mathcal{L}} \in [0,1] \label{eq:safeIntervals}
\end{align}
\end{subequations}
\end{proposition}
\begin{proof}[Proof.]
Consider $\mathcal{\hat V}_S = \{(w,z) \in \mathcal{G'} \times \mathcal{L'}: \eqref{SOCP_Chance}, \eqref{constr:compact_totalMaintenanceNumberMP}\}$. Note that by the independence assumption, set $\mathcal{\hat V}$ is equivalent to:
\begin{align*}
    \mathcal{\hat V} = \{(w,z) \in \{0,1\}^{|\mathcal{G'}| \times |\mathcal{L'}|}: \mathbb{P} ( \sum_{i \in \mathcal{G}} \sum_{t \in \mathcal{\bar T}} \zeta_{it} w_{it} \le \rho_{\mathcal{G}}) \mathbb{P}(\sum_{i \in \mathcal{L}} \sum_{t \in \mathcal{\bar T}} \zeta_{ijt} z_{ijt} \le \rho_{\mathcal{L}}) \ge 1-\alpha, \ \eqref{constr:compact_totalMaintenanceNumberMP}\}
\end{align*}
To show that $\mathcal{\hat V}_S \subseteq \mathcal{\hat V}$, we let $(\tilde w, \tilde z) \in \mathcal{\hat V}_S.$ As proven in \citet{Basciftci2018}, we have $\mathbb{P} ( \sum_{i \in \mathcal{G}} \sum_{t \in \mathcal{\bar T}} \zeta_{it} \tilde w_{it} \ge \rho_{\mathcal{G}}) \le \frac{\sum_{i \in \mathcal{G}} \sum_{t \in \mathcal{\bar T}} \mathbb{E}[\zeta_{it}] \tilde w_{it}}{\rho_{\mathcal{G}}}$. By using \eqref{eq:safeLinearGen}, we obtain $\mathbb{P} ( \sum_{i \in \mathcal{G}} \sum_{t \in \mathcal{\bar T}} \zeta_{it} \tilde w_{it} \le \rho_{\mathcal{G}}) \ge \bar \alpha_{\mathcal{G}}$. Similarly, we also have $\mathbb{P} ( \sum_{i \in \mathcal{L}} \sum_{t \in \mathcal{\bar T}} \zeta_{ijt} \tilde z_{ijt} \le \rho_{\mathcal{L}}) \ge \bar \alpha_{\mathcal{L}}$. By combining these results with relations \eqref{eq:safeSOCP} and \eqref{eq:safeIntervals}, we have $\mathbb{P} ( \sum_{i \in \mathcal{G}} \sum_{t \in \mathcal{\bar T}} \zeta_{it} \tilde w_{it} \le \rho_{\mathcal{G}}) \mathbb{P} ( \sum_{i \in \mathcal{L}} \sum_{t \in \mathcal{\bar T}} \zeta_{ijt} \tilde z_{ijt} \le \rho_{\mathcal{L}}) \ge 1- \alpha$ proving that $(\tilde w, \tilde z) \in \mathcal{\hat V}$.
\Halmos
\end{proof}
In this formulation, the variables $\bar \alpha_{\mathcal{G}}$ and  $\bar \alpha_{\mathcal{L}}$ are used to represent the probabilities in the joint chance-constraint \eqref{eq:chanceBernoulli}. Inequalities \eqref{eq:safeLinearGen} and \eqref{eq:safeLinearLine} are affine and expectations are efficiently computable (since the random vector $\zeta$ consists of Bernoulli random variables), also \eqref{eq:safeSOCP} is an SOCP constraint so that the proposed deterministic safe approximation of \eqref{chanceConstr} is convex and tractable. Note that this approximation may be too conservative in some cases entailing an early maintenance planning when $\alpha$ gets smaller.

\subsection{Optimality Cut Families} \label{s:optimalityCutFamilies}
In this section, we introduce various sets of optimality cuts which are generated by \texttt{OptimalityCut($v, \xi, L$)} in Algorithm \ref{alg:decomposition}. In Section \ref{ss:integerL-ShapedOptimalityCuts}, we specify the well-known classical integer L-shaped optimality cuts introduced by \citet{Laporte1993}. We introduce new optimality cuts by strengthening the classical integer L-shaped optimality cuts in Section \ref{ss:newOptimalityCuts}. We explain the rationale behind these optimality cuts in detail and provide the proofs of their validity and strength. 

\subsubsection{Integer L-Shaped Optimality Cuts} \label{ss:integerL-ShapedOptimalityCuts}
%Recall that we assume that the two-stage joint chance-constrained stochastic program \eqref{eq:twoStageCompact} has relatively complete recourse, thus, we only focus on generating optimality cuts. 
The idea of the integer L-shaped method is to approximate the expected recourse function by adding optimality cuts as the supporting hyperplanes of $\mathcal{Q}(v, \xi)$. These cuts depend on a given maintenance decision $v^{(r)} \in \mathcal{\hat V}$ at some iteration $r$ and a realization of the random vector $\xi$, and gradually reduce the feasible region defined in the $(v, \theta)$-space. We first introduce the index set at iteration $r$ as $V_h(v^{(r)}) := \{t \in \mathcal{\bar T}: v_{ht}^{(r)} = 1\}$ for every $h \in \mathcal{H}'$
%For every $v \in \mathcal{\hat V}$, the maintenance decision $v_h$ for component $h$ is a binary vector which implies that $V_h^{(r)}$ is a singleton. 
and define $\mathcal{Q}(v^{(r)}, \xi)$ as the expected second-stage value. Note that the set $V_h(v^{(r)})$ is a singleton by constraints \eqref{limitOnGenMaintenance} and \eqref{limitOnLineMaintenance}. The structure of the stochastic program with independent scenarios allows us to define multi-cuts. Given a feasible maintenance decision $v^{(r)} \in \mathcal{\hat V}$, the classical integer L-shaped optimality cut for subproblem $k$ added to the master problem at the iteration $r$ is defined as:
%
% CLASSICAL L-SHAPED ONE-CUT
%
\begin{equation} \label{eq:classical_Lshaped}
    \theta^k \ge (\mathcal{Q}(v^{(r)},\xi_k) - L^k) 
    \sum_{h \in \mathcal{H}'} \big(\sum_{t \in V_h(v^{(r)})} ( v_{ht}-1) - \sum_{t \notin V_h(v^{(r)})} v_{ht} \big) + \mathcal{Q}(v^{(r)}, \xi_k),
\end{equation}
where $L^k$ is a valid lower bound on the expected second-stage value of subproblem $k$.
%The set of optimality cuts given by \eqref{eq:classical_Lshaped} is valid and moreover, since $\mathcal{\hat V}$ is finite with the assumption that stochastic optimization problem is feasible and bounded below, the integer L-shaped method terminates with an optimal solution in a finite number of steps \citet{Laporte1993}. 
We can also obtain the single-cut version of the optimality cut by summing over all scenarios on both sides of the relation \eqref{eq:classical_Lshaped}:
\begin{equation} \label{eq:classical_Lshaped_single}
    \sum_{k \in \mathcal{K}} \theta^k \ge \sum_{k \in \mathcal{K}} (\mathcal{Q}(v^{(r)},\xi_k) - L^k) 
    \sum_{h \in \mathcal{H}'} \big(\sum_{t \in V_h(v^{(r)})} ( v_{ht}-1) - \sum_{t \notin V_h(v^{(r)})} v_{ht} \big) + \sum_{k \in \mathcal{K}} \mathcal{Q}(v^{(r)}, \xi_k).
\end{equation}
For ease of notation, we will not carry the superscript $(r)$ in the remainder of this section.
\subsubsection{New Optimality Cuts} \label{ss:newOptimalityCuts}
Next, we introduce a new set of optimality cuts by adapting the multi-cut version of the integer L-shaped optimality cut  \eqref{eq:classical_Lshaped} to better approximate the expected recourse function $\mathcal{Q}$.
\begin{proposition} \label{prop:optCut_K+}
Given a feasible maintenance decision $v^* \in \mathcal{\hat V}$, the following set of optimality cuts is valid and stronger than the classical L-shaped optimality cut \eqref{eq:classical_Lshaped}:
%
% STRONG K
%
\begin{equation} \label{eq:optCut_K+}
\setlength{\jot}{8pt}
\theta^k \ge (\mathcal{Q}(v^*,\xi_k) - L^k) 
\sum_{h \in \mathcal{H}'} \big(\sum_{t \in V_h(v^*)} v_{ht}-1\big) + \mathcal{Q}(v^*, \xi_k).
\end{equation}
\end{proposition}
\begin{proof}[Proof.]
Suppose we are given a maintenance decision $v^* \in \mathcal{\hat V}$. We consider the following quantity $Q_h := \sum_{t \in V_h(v^*)} v_{ht}$. If $Q_h = 1$ for every $h \in \mathcal{H}'$, then the cut in \eqref{eq:optCut_K+} becomes $\theta^k \ge \mathcal{Q}(v^*, \xi_k).$  If $Q_h = 0$ for some $h \in \mathcal{H}'$, then we have $\sum_{h \in \mathcal{H}'} Q_h - |\mathcal{H}'| \le -1$.  In this case, the optimality cut \eqref{eq:optCut_K+} becomes redundant since $\theta^k \ge L^k + A$ where $A \le 0$. To prove the strength of the cut, we also consider the following quantity $\bar Q_h := \sum_{t \notin V_h(v^*)} v_{ht}$ such that $\bar Q_h \in \{0,1\}$ for every $h \in \mathcal{H}'$. Then, clearly we have $Q_h - \bar Q_h - |V_h(v^*)| \le Q_h - |V_h(v^*)|$ for every $h \in \mathcal{H}'$. This implies that \eqref{eq:optCut_K+} is stronger than \eqref{eq:classical_Lshaped}.
\Halmos
\end{proof}
We obtain the single-cut version of \eqref{eq:optCut_K+} by summing over all scenarios on both sides:
\begin{equation} \label{eq:eq:optCut_K+_single}
    \sum_{k \in \mathcal{K}} \theta^k \ge  \sum_{k \in \mathcal{K}} (\mathcal{Q}(v^*,\xi_k) - L^k) 
\sum_{h \in \mathcal{H}'} \big(\sum_{t \in V_h(v^*)} v_{ht}-1\big) +  \sum_{k \in \mathcal{K}} \mathcal{Q}(v^*, \xi_k).
\end{equation}

Next, we explain how to derive even stronger optimality cuts than \eqref{eq:optCut_K+}. The key idea of deriving such optimality cuts is to identify a set of maintenance decisions which will yield  the same operational cost. Given a maintenance decision $v^* \in \mathcal{\hat V}$, we introduce the set $ R^k(v^*)$ as the set of all feasible maintenance decisions which will have the same second-stage value under scenario $k$:
\[ R^k (v^*) = \{ v \in \mathcal{\hat V}: \ \mathcal{Q}(v, \xi_k) = \mathcal{Q}(v^*, \xi_k)\}.\]
Further, we can define $\mathcal{\hat T}^k_{h}(v^*)$ as the set of maintenance period indices of each component $h$ under scenario $k$ such that maintaining component $h$ in period $t$ will yield to the same operational cost for every $t \in \mathcal{\hat T}^k_{h}(v^*)$:
\[\mathcal{\hat T}^k_{h} (v^*) = \{ t\in\mathcal{\bar T} : \ \exists v \in R^k (v^*) \text{ such that } v_{ht} = 1 \}.\]
%We use set $V_h$ in generating the previous optimality cuts for every scenario $k \in \mathcal{K}$. In other words, set $V_h$ is independent from scenario subproblems. This results in adding almost identical cuts for every scenario $k$ which may not be very interesting, thus, we can define better optimality cuts which are unique for scenario $k \in \mathcal{K}$. The set $\mathcal{\hat T}^k_h$ helps us identifying those maintenance periods in which for every maintenance occurring for component $h$, we will obtain the same second-stage value under scenario $k$.

\begin{proposition} \label{prop:optCut_K++}
Given a feasible maintenance decision $v^* \in \mathcal{\hat V}$, the following set of optimality cuts is valid and stronger than the set of optimality cuts given by \eqref{eq:optCut_K+}: 
\begin{equation} \label{eq:optCut_K++}
\theta^k \ge (\mathcal{Q}(v^*,\xi_k) - L^k) 
\sum_{h \in \mathcal{H}'} \big ( \sum_{t \in \mathcal{\hat T}^k_{h}(v^*) } v_{ht}-1 \big) + \mathcal{Q}(v^*, \xi_k).
%\end{align}
\end{equation}
\end{proposition}
\begin{proof}[Proof.]
%Given a maintenance decision $v^* \in \mathcal{\hat V}$, we consider the quantity $Q_h \equiv \sum_{t \in \mathcal{\hat T}^k_{h}(v^*) } v_{ht}$ with $Q_h \in \{0,1\}$ for every $h \in \mathcal{H}'$. If $Q_h = 1$, then the cut in \eqref{eq:optCut_K++} becomes $\theta^k \ge \mathcal{Q}(v^*, \xi_k)$. If $Q_h = 0$ for some $h \in \mathcal{H}'$, then the cut in \eqref{eq:optCut_K+} is not generated for the maintenance decision $v$. Therefore, we have $\sum_{h \in \mathcal{H}'} Q_h - |\mathcal{H}'| \le -1$. In this case, the optimality cut \eqref{eq:optCut_K++} becomes redundant since $\theta^k \ge L^k + A$ where $A \le 0$. 
Given a maintenance decision $v^* \in \mathcal{\hat V}$, consider the following quantity $Q_h := \sum_{t \in \mathcal{\hat T}^k_{h}(v^*)} v_{ht}$. By constraint \eqref{constr:compact_totalMaintenanceNumberMP}, we know that $Q_h \in \{0,1\}$ for every $h \in \mathcal{H'}$, then the proof of validity follows as in Proposition \ref{prop:optCut_K+}.
To prove the strength of the cut, let the maintenance decision of component $h$ under $v^*$ be in maintenance period $t'$, which clearly implies that $t' \in \mathcal{\hat T}_h^k(v^*)$ for every $k \in \mathcal{K}$ and $V_h(v^*) = \{t'\}$. Then, we have the following relation $\sum_{t \in \mathcal{\hat T}^k_{h}(v^*)} v_{ht}\ge \sum_{t \in V_{h}(v^*)} v_{ht}$ since $V_h(v^*) \subseteq \mathcal{\hat T}_h^k(v^*)$ holds for every $k \in \mathcal{K}$. This implies that \eqref{eq:optCut_K++} is stronger than \eqref{eq:optCut_K+}.
\Halmos
\end{proof}
Given a maintenance decision $v^* \in \mathcal{\hat V}$, obtaining set $R^k(v^*)$ for every $k \in \mathcal{K}$ might be computationally expensive; however, in our setting, we can obtain a subset of $\mathcal{\hat T}^k_{h}(v^*)$ by identifying whether each component $h \in \mathcal{H'}$ enters predictive or corrective maintenance depending on decision $v^*$ and the failure times under scenario $k$. In particular, if component $h$ is scheduled for predictive maintenance under scenario $k$, we define this subset as the period that this component is scheduled for maintenance. On the other hand, if component $h$ enters corrective maintenance under scenario $k$, this subset consists of all maintenance periods from the failure time of component $h$ to the end of the planning horizon. 
\begin{corollary}
Given $v^* \in \mathcal{\hat V}$ and a subset $\mathcal{\hat T}'(v^*) \subseteq \mathcal{\hat T}^k_{h}(v^*)$, the following set of optimality cuts is valid and stronger than the set of optimality cuts given by \eqref{eq:optCut_K+}: 
\begin{equation} 
\theta^k \ge (\mathcal{Q}(v^*,\xi_k) - L^k) 
\sum_{h \in \mathcal{H}'} \big ( \sum_{t \in \mathcal{\hat T}'(v^*)} v_{ht}-1 \big) + \mathcal{Q}(v^*, \xi_k).
%\end{align}
\end{equation}

\end{corollary}
As before, we can obtain the single-cut version of \eqref{eq:optCut_K++} by summing over all scenarios on both sides:
\begin{equation} \label{eq:eq:optCut_K++_single}
    \sum_{k \in \mathcal{K}} \theta^k \ge  \sum_{k \in \mathcal{K}} (\mathcal{Q}(v^*,\xi_k) - L^k) 
\sum_{h \in \mathcal{H}'} \big ( \sum_{t \in \mathcal{\hat T}^k_{h}(v^*) } v_{ht}-1 \big) +  \sum_{k \in \mathcal{K}} \mathcal{Q}(v^*, \xi_k).
\end{equation}

Next, we explain how to derive a different set of optimality cuts by exploiting the status idea (explained in Section \ref{ss:decompositionAlgorithm}). We first provide an overview of the idea on how to derive these alternative optimality cuts. After obtaining a maintenance decision by solving \eqref{eq:masterProblem}, we observe that there is no coupling constraint between maintenance periods in scenario subproblems. This allows us to obtain even smaller subproblems by decomposing with respect to independent maintenance periods. We refer to this property as time-decomposability of scenario subproblems before and the formulation of these subproblems are introduced in \eqref{eq:recourseFunctionDecomposed}. By using this property, we can rewrite \eqref{eq:optCut_K+} in the following form:
\begin{equation} \label{eq:optCut_KT+}
\setlength{\jot}{8pt}
\theta^k_t \ge (\mathcal{Q}_t(v^*,\xi_k) - L^k_t) 
\sum_{h \in \mathcal{H}'} \big(\sum_{t \in V_h(v^*)} v_{ht}-1\big) + \mathcal{Q}_t(v^*, \xi_k).
\end{equation}
Recall that given a maintenance decision $v^*$, we define the status vector for maintenance period $t$, denoted by $u_t^k(v^*)$, representing the availability of all components under scenario $k$. This property allows us to represent each scenario subproblem with respect to their status vectors and as a consequence, we restrict ourselves to the scenario subproblems such that corresponding status vectors are all unique. By considering these unique status vectors, we can obtain even stronger optimality cuts than \eqref{eq:optCut_KT+}. Given a maintenance decision $v^*$, we first define $\mathcal{\tilde T}_{ht}^k (v^*)$ as the set of periods such that maintaining a component $h$ in period $t$ will yield the same status $u_{ht}^k(v^*)$ under scenario $k$:
\begin{equation*}
    \mathcal{\tilde T}_{ht}^k (v^*) := \big\{ t'\in\mathcal{\bar T}: \exists v \in \mathcal{\hat V} \text{ such that } v_{ht'} = 1, \ u_{ht}^k (v^*) = u_{ht}^k (v) \big\}.
\end{equation*}
After identifying the status for component $h \in \mathcal{H'}$ in maintenance period $t \in \mathcal{\bar T}$ (explained in Section \ref{ss:decompositionAlgorithm}), we can easily obtain a subset of $\mathcal{\tilde T}_{ht}^k (v^*)$ under each scenario $k \in \mathcal{K}$.

\begin{proposition} \label{prop:optCut_KT+++}
Given a feasible maintenance decision $v^* \in \mathcal{\hat V}$, the following set of optimality cuts is valid and stronger than then the set of optimal cuts in \eqref{eq:optCut_KT+}:
\begin{equation} \label{eq:optCut_KT+++}
\setlength{\jot}{8pt}
\theta^k_{t} \ge (\mathcal{Q}_{t}(v^*,\xi_k) - L^k_{t}) 
\sum_{h \in \mathcal{H}'} \big ( \sum_{t' \in \mathcal{\tilde T}_{ht}^k(v^*) } v_{ht'}-1 \big) + \mathcal{Q}_{t}(v^*, \xi_k).
%\end{align}
\end{equation}
\end{proposition}
\begin{proof}[Proof.]
%Given a maintenance decision $v^* \in \mathcal{\hat V}$, we consider the quantity $Q_h \equiv \sum_{t \in \mathcal{\hat T}^k_{ht}(v^*) } v_{ht}$ with $Q_h \in \{0,1\}$ for every $h \in \mathcal{H}'$. If $Q_h = 1$, then the cut in \eqref{eq:optCut_KT+++} becomes $\theta^k_t \ge \mathcal{Q}_t(v^*, \xi_k)$. If $Q_h = 0$ for some $h \in \mathcal{H}'$, then the cut in \eqref{eq:optCut_K+} is not generated for the maintenance decision $v$. Therefore, we have $\sum_{h \in \mathcal{H}'} Q_h - |\mathcal{H}'| \le -1$.  In this case, the optimality cut \eqref{eq:optCut_K++} becomes redundant since $\theta^k_t \ge L^k_t + A$ where $A \le 0$. 
The proof of validity is similar as in Proposition \ref{prop:optCut_K++}. To prove the strength of the cut, let the maintenance decision of component $h$ under $v^*$ be in maintenance period $t'$, which clearly implies that $t' \in \mathcal{\tilde T}_{ht}^k(v^*)$ for every $k \in \mathcal{K}$ and $t \in \mathcal{T}$, and $V_h(v^*) = \{t'\}$. Then, we have the following relation $\sum_{t \in \mathcal{\tilde T}^k_{ht}(v^*)} v_{ht}\ge \sum_{t \in V_{h}(v^*)} v_{ht}$ since $V_h(v^*) \subseteq \mathcal{\tilde T}_{ht}^k(v^*)$ holds for every $k \in \mathcal{K}$ and $t \in \mathcal{T}$. This implies that \eqref{eq:optCut_KT+++} is stronger than \eqref{eq:optCut_KT+}.
\Halmos
\end{proof}

\begin{corollary}
Given $v^* \in \mathcal{\hat V}$ and a subset $\mathcal{\tilde T}'(v^*) \subseteq \mathcal{\tilde T}^{k}_{ht}(v^*)$, the following set of optimality cuts is valid and stronger than then the set of optimal cuts in \eqref{eq:optCut_KT+}:
\begin{equation} 
\setlength{\jot}{8pt}
\theta^k_{t} \ge (\mathcal{Q}_{t}(v^*,\xi_k) - L^k_{t}) 
\sum_{h \in \mathcal{H}'} \big ( \sum_{t' \in \mathcal{\tilde T}'(v^*)} v_{ht'}-1 \big) + \mathcal{Q}_{t}(v^*, \xi_k).
%\end{align}
\end{equation}
\end{corollary}
%\citet{Laporte1993} proved the property of finite convergence of the integer L-shaped method when the first-stage variables are pure binary. In our case, Proposition \ref{prop:finiteConvergence} follows immediately since our decomposition algorithm (\ref{alg:decomposition}) is an enhanced version of integer L-shaped method with various algorithmic enhancements.
We conclude this section by proving the property of finite convergence of our decomposition algorithm (Algorithm \ref{alg:decomposition}).

\begin{proposition} \label{prop:finiteConvergence}
Algorithm \ref{alg:decomposition} converges in finitely many iterations. 
\end{proposition}

\begin{proof}[Proof.]
We observe that there are only finitely many feasible first-stage decisions since each maintenance decision is pure binary. In view of this observation and the integer L-shaped algorithm, when the safe approximation of the joint chance-constraint is used, we can add finitely many optimality cuts that lead to the convergence of Algorithm \ref{alg:decomposition} in finitely many iterations. When the exact representation of the joint chance-constraint is used, we can add a finite number of violated cover inequalities that can be identified through Algorithm \ref{alg:representChance} since this constraint includes only the first-stage decisions. 
\Halmos
\end{proof}

\subsection{Flow Limit Analysis} \label{s:flowLimitAnalysis}
% The set of such constraints is called redundant if and only if the feasible region remains the same without imposing this set on the resulting optimization problem. 
Optimization problems in power systems may involve redundant transmission flow limits. When this redundancy is identified and handled efficiently, computational requirements for solving the optimization problem may potentially decrease. Given a power demand vector $\bar d$, the following relaxation of the operational subproblem can be used to identify such redundancy:
%Respective optimization-based methods (\citet{eliminationComposite}, \citet{eliminationSCUC}, \citet{eliminationNCUC}) and analytical methods (\citet{eliminationAnalytic}, \citet{eliminationDCOPF}) have been studied in the literature. 
 %For each transmission line $(i,j) \in \mathcal{L}''$, we solve the following linear program:
\begin{subequations} \label{eq:discardLineModel}
\begin{align}
f^*_{i',j'}(\bar d) = \max_{{ f, y, x, \delta, p, q, d }} &\hspace{0.5em} f_{i'j'}\\ 
\mathrm{s.t.}
&\hspace{0.5em} q_i \le d_i \le \bar d_i &i& \in \mathcal{B} \label{constr:discardWorst}\\
&\hspace{0.5em} \sum_{i' \in \mathcal{G}(i)} p_{i'} + q_{i} -d_{i} = \sum_{j \in \delta^+(i)} f_{ij} - \sum_{j \in \delta^-(i)} f_{ji}  & i& \in \mathcal{B} \label{constr:discardFlowBalance}\\ 
&\hspace{0.5em} B_{ij}(\delta_{i}-\delta_{j}) = f_{ij}  &(&i,j) \in \mathcal{L}'' \label{constr:discardFlowDef} \\
& \hspace{0.5em} B_{ij}(\delta_{i}-\delta_{j}) - M_{ij}(1 - y_{ij}) \le f_{ij} && \nonumber\\
&\hspace{1.5em} \le B_{ij}(\delta_{i}-\delta_{j}) + M_{ij}(1 - y_{ij}) &(&i,j) \in \mathcal{L'} \label{constr:discardVoltage}\\
%& \hspace{0.5em} -\Bar{f}_{ij} \le  f_{ij}  \le \Bar{f}_{ij} &(&i,j) \in \mathcal{\tilde L} \label{constr:discardFlowLimits}\\
& \hspace{0.5em} -\Bar{f}_{ij}y_{ij} \le  f_{ij}  \le \Bar{f}_{ij}y_{ij} &(&i,j) \in \mathcal{L'} \label{constr:discardFlowLimitsL}\\
& \hspace{0.5em} p_i^{min}x_{i}  \le p_{i} \le p_i^{max}x_{i}     &i& \in \mathcal{G} \label{constr:discardPowerGen}\\
&\hspace{0.5em} y \in [0,1]^{|\mathcal{L'}|},  \ x \in [0,1]^{|\mathcal{G}|}, \delta \in [\delta^{\min}, {\delta}^{\max}], q \ge 0.
\end{align}
\end{subequations}
In this formulation, the decision variables $f, y, x, p, \delta, q$ represent power flow, switching status of transmission lines, commitment status and power generation of generators, voltage angle and demand curtailment of buses, respectively. We also define another continuous decision variable $d_i$ which represents power demand for $i \in \mathcal{B}$. Constraint \eqref{constr:discardWorst} ensures that $d_i$ remains feasible for the operational subproblems. The remaining constraints are operational constraints implied by the power network. We note that \eqref{eq:discardLineModel} is a relaxation of the original operational problem, since switching variables are considered as continuous and start-up, shut-down restrictions are omitted by focusing the analysis on a single period. 

Recently, \citet{Basciftci2018} solve a relaxation of the operational subproblem to identify redundant flow limits by considering the peak demand of each bus within the planning horizon. Their model (referred to as \textit{FlowModel-I} in our paper) is similar to \eqref{eq:discardLineModel} when $\bar d = [\max_{t \in \mathcal{T}, s \in \mathcal{S}}\{d_{its}\}; i \in \mathcal{B}]$. 
%However, the authors do not consider the time-decomposability of operational subproblems. 
By utilizing the time-decomposability of operational subproblems, we improve their model (referred to as \textit{FlowModel-II}) by replacing $\bar d$ with $\bar d_t = [\max_{s \in \mathcal{S}}\{d_{its}\}; i \in \mathcal{B}]$ and solve \eqref{eq:discardLineModel} for every $t \in \mathcal{T}$. Similarly, we can easily improve this model (referred to as \textit{FlowModel-III}) by replacing $\bar d$ with the actual power demand in hourly subperiod $s$ of maintenance period $t$, that is, $\bar d_{ts} = [d_{its}; i \in \mathcal{B}]$ and solve \eqref{eq:discardLineModel} for every $t \in \mathcal{T}$ and $s \in \mathcal{S}$. Given a transmission line $(i',j')$ and power demand $\bar d$, suppose we solve the linear program $\eqref{eq:discardLineModel}$ and obtain an optimal solution $f^*_{i'j'}(\bar d)$. If $f^*_{i'j'}(\bar d)$ is strictly less than $\bar f_{i'j'}$, we ensure that flow upper limit corresponding for transmission line $(i', j')$ will not be violated which allows us to eliminate the corresponding constraint from the optimization model \eqref{eq:optimizationModel}. Otherwise, we impose this constraint for transmission line $(i',j')$. Similarly, we can also identify redundant lower flow limits by changing the objective function of \eqref{eq:discardLineModel} with $-f_{i'j'}(\bar d)$. To this end, the number of constraints that can be eliminated depending on the choice of the demand parameter, and how many times the corresponding model is solved.

\subsection{Sample Average Approximation} \label{s:sampleAverageApproximation}
%The size of the set of all possible scenarios grows exponentially in the number of components, i.e., $|\mathcal{\bar T}|^{|\mathcal{H'}|}$. Although this set itself is finite, enumerating  many scenarios within a reasonable amount of time is computationally arduous for solving the true problem \eqref{eq:twoStageCompact}. To address this issue, we adapt the SAA algorithm, introduced by \citet{SAA_Kleywegt}, to solve the SAA problem of the true counterpart.
Since the number of scenarios of our stochastic program grows exponentially fast in the number of system components considered for maintenance, i.e., $|\mathcal{\bar T}|^{|\mathcal{H'}|}$, solving this program becomes computationally more demanding as the instance size increases. Thus, we solve this problem using the SAA algorithm (Algorithm \ref{alg:SAA} in Appendix \ref{app:SAA}).
%we adapt the SAA algorithm, introduced by \citet{SAA_Kleywegt} to solve \eqref{eq:twoStageCompact}. 
%Next, we state Algorithm \ref{alg:SAA} for the joint chance-constrained stochastic optimization problem \eqref{eq:optimizationModel}. 
In our setting, the set of training scenarios are generated over the components $\mathcal{H'}$, whereas the set of test scenarios are generated over the all set of components $\mathcal{H}$ to evaluate the true performance of the proposed approach. We first generate SAA replications of size $M$, each consisting of independent and identically distributed (i.i.d.) failure scenarios of size $N$. We solve the corresponding SAA problem for each replicate and obtain their optimal values and $\epsilon$-optimal solutions. By averaging these optimal values, we obtain the mean estimate for the true lower bound. We later evaluate each $\epsilon$-optimal solution over a sample size of $N'$ with $N' \gg N$ and choose the best candidate solution among all $\epsilon$-optimal solutions by setting the corresponding objective value as the best upper bound estimate. In Step \ref{alg:SAA_CIofUB} and Step \ref{alg:SAA_CIofLB} of Algorithm \ref{alg:SAA}, we construct the upper and lower statistical bounds, developed by \citet{Mak99}, to assess the quality of the optimal solution produced by the SAA algorithm, respectively. These statistical bounds are used to construct confidence intervals (CIs) for estimating the optimality gap between the optimal value produced by Algorithm \ref{alg:SAA} and the optimal value of the true problem. %After obtaining the upper and lower statistical bounds, an approximate $(1-\alpha)$ level confidence interval for the true objective value is constructed by $(\hat \mu_L - t_{\alpha/2, M-1} \hat \sigma_L, \ \hat \mu_U + z_{\alpha/2} \hat \sigma_U)$.
%
%
%
% ALGORITHM SAA
%
%
%
%Note that \citet{SAA_Kleywegt} demonstrate that the event that the optimal solution produced by the SAA problem will be an exact optimal solution for the true optimal solution happens with probability approaching one exponentially fast as the number of scenarios increases. The authors also argue about the trade-off between the quality of the solution produced by the SAA algorithm and the computational complexity as the sample size increases, thus, the sample size of each replicate can be chosen based preliminary computations.

\section{Computational Experiments} \label{s:computationalExperiments}
To demonstrate the computational performance and efficiency of the proposed algorithm, we conduct an extensive computational study on various modified IEEE instances from MATPOWER \citep{MATPOWER}. In Section \ref{ss:ExperimentalSetup}, we explain the experimental setup in detail. In Section \ref{ss:performanceOfTheProposedAlgorithm}, we show the computational efficiency of our algorithmic enhancements and sets of optimality cuts with parallelization in comparison with the state-of-the-art solver \texttt{GUROBI}. We provide the statistical results on the true optimal value produced by the SAA algorithm with different sizes of failure scenarios in Section \ref{ss:SAAResults}. We evaluate the quality of the maintenance schedules obtained by the proposed stochastic models in Section \ref{ss:modelComparison}. Lastly, we investigate the effects of the cardinality of the sets $\mathcal{G'}$ and $\mathcal{L'}$ in Section \ref{ss:sensitivityAnalysis}.

\subsection{Experimental Setup} \label{ss:ExperimentalSetup}

\subsubsection{Instance Creation}
For our joint chance-constrained stochastic program model, we consider a one-week planning period with daily maintenance decisions and hourly operational decisions. The planning horizon starts on a Monday at $00:00$. We obtain the weekly electricity consumption data available from the U.S. Energy Information Administration \citep{EIA} since actual power demand parameters in standard IEEE instances are given for an hourly period for each bus. We later use this data to generate a new power demand dataset through normalization such that $\{d_{its}; t \in \mathcal{T}, s \in \mathcal{S}\}$ follows a similar trend for every $i \in \mathcal{B}$. Since corrective maintenance is undesirable and unexpected, it is more expensive and takes longer amount of time compared to predictive maintenance.  Specifically, we assume that the maintenance durations are $\tau^{p}_{{\scriptscriptstyle \mathcal{G}}} = \tau^{p}_{{\scriptscriptstyle \mathcal{L}}} = 1$ and $\tau^{c}_{{\scriptscriptstyle \mathcal{G}}} = \tau^{c}_{{\scriptscriptstyle \mathcal{L}}} = 2$ days. We also assume that maintenance cost for generators is a function of generation cost and generation capacity. In particular, we let $C_i^p = \bar p_i c_i |\mathcal{S}|$ for $i \in \mathcal{G}$. %In other words, we assume that daily predictive maintenance costs for generators are according to their daily maximum loss production. 
Additionally, we let $C_i^c = 3 C_i^p$ for $i \in \mathcal{G}$, $C_{ij}^p = 0.1 \sum_{i \in G} C_i^p / |\mathcal{G}|$ and $C_{ij}^c = 3 C_{ij}^p$ for $(i,j) \in \mathcal{L}$. We have chosen the constant $M_{ij}$ sufficiently large for $(i,j) \in \mathcal{L'}$ such that constraint \eqref{linearizedAngle} becomes redundant when $y_{ijts}^k = 1$. In particular, we let $M_{ij} = B_{ij}({\delta}_i^{\max} -{\delta}_j^{\min})$ for $(i,j) \in \mathcal{L'}$ (e.g., see \citet{OTS-bigM}). We choose the probability thresholds $p_{fail}^{\mathcal{G}} = 0.1$ and $p_{fail}^{\mathcal{L}} = 0.2$ for generators and transmission lines, respectively. We then identify those system components prone to failure within the planning horizon as explained in Section \ref{ss:problemSetting}. For the computational experiments subject to the joint chance-constraint \eqref{constr:compact_jointChance}, the thresholds $\rho_{\mathcal{G}}$ and $\rho_{\mathcal{L}}$ are set to $1$ and $\max\{1, \lfloor | \mathcal{L}| / 20 \rfloor\}$, respectively. These experiments are conducted with a probability level $\alpha = 0.1$ of the joint chance-constraint. We report the cardinality of the subsets of $\mathcal{G}$ and $\mathcal{L}$, and the threshold parameters of the joint chance-constraint for each instance in Table \ref{table:cardinalityAndRhos}.
\begin{table}[H]
\centering
\begin{tabular}{rcccccc}
\hline
& \multirow{2}{*}{$|\mathcal{G'}|$} & \multirow{2}{*}{$|\mathcal{G''}|$} & \multirow{2}{*}{$|\mathcal{L'}|$} & \multirow{2}{*}{$|\mathcal{L''}|$} & \multirow{2}{*}{$\rho_{\mathcal{G}}$} & \multirow{2}{*}{$\rho_{\mathcal{L}}$} \\
          &                                   &                                    &                                   &                                    &                                       &                                       \\ \hline
$9$-bus   & 1                                 & 2                                  & 3                                 & 6                                  & 1                                     & 1                                     \\
$39$-bus  & 4                                 & 6                                  & 4                                 & 42                                 & 1                                     & 2                                     \\
$57$-bus  & 2                                 & 5                                  & 7                                 & 73                                 & 1                                     & 4                                     \\
$118$-bus & 4                                 & 15                                 & 9                                 & 177                                & 1                                     & 9                                     \\ \hline
\end{tabular}
\caption{Cardinality of Sets and Threshold Parameters.}
\label{table:cardinalityAndRhos}
\end{table}
We generate a dataset consisting of unique degradation signals due to the lack of publicly available data to estimate the parameters of the prior distributions of $\upsilon_h$ and $\beta_h$ for $h \in \mathcal{H}$. In power systems, it is realistic to assume that generators are more likely to fail than transmission lines (see, for example, \citet{Papa2015}). Thus, we follow this assumption with our dataset. For simplicity, we assume that the variance of $\upsilon_h$ and $\beta_h$ are indeed known and held constant over the planning horizon for $h \in \mathcal{H}$. Therefore, we are only interested in estimating the prior mean of $\upsilon_h$ and $\beta_h$, denoted by $\mu_0$ and $\mu_1$, respectively. First, we focus on estimating $\mu_0$ and $\mu_1$ among the set of generators. For that purpose, we generate $100$ unique degradation signals. Let us label these degradation signals with an index $j$ where $j = 1, \dots, 100$. We assume that degradation signal $j$ has the functional form \eqref{eq:degradationSignal} with $\upsilon_j \sim \mathcal{N} (20, 10^2)$ and $\beta_j \sim \mathcal{N} (5, 0.3^2)$ and $\sigma_j = 3$ for $j = 1, \dots, 100$. The degradation signal threshold $\Lambda$ is set to $100$. We observe degradation signal $j$ at discrete time points until a failure time $\xi_j = \{t : D_j(t) \ge 100, \ t \ge 0\}$ for $j = 1, \dots, 100$. We remind the reader that $D_j^i$ is defined as the increment of degradation signals between times $t_j^i$ and $t_j^{i-1}$ for $i = 2, \dots, \xi_j$ where $D_j^1 = D_j(1)$, for $j=1,\dots, 100$. We find the point estimate of $\mu_0$ with $\sum_{j = 1}^{100} D^1_j / 100$. To obtain the point estimate of $\mu_1$, we first compute the prior mean estimate of $\beta_j$ as $ \hat{\mu}_j = ( \sum_{i=1}^{\xi_j}D^i_j - D^1_j) / \xi_j$ for $j = 1,\dots, 100$. Then, we find the point estimate of $\mu_1$ with $\sum_{j = 1}^{100} \hat{\mu}_j / 100$. Eventually, we obtain the prior mean estimate among the set of generators. Secondly, to estimate $\mu_0$ and $\mu_1$ among the set of transmission lines, we follow a similar procedure after generating $100$ unique degradation signals with $\upsilon_j \sim \mathcal{N} (15, 5^2)$ and $\beta_j \sim \mathcal{N} (3, 0.3^2)$ and $\sigma_j = 1$ for $j = 1, \dots, 100$. Finally, we obtain the prior mean estimates of the stochastic parameters $\upsilon_h$ and $\beta_h$ of the degradation signal model for every $h \in \mathcal{H}$.

Next, we obtain the posterior distribution of the unknown parameters of $\upsilon_h$ and $\beta_h$ for $h \in \mathcal{H}$ with a Bayesian approach given the recently observed real-time condition-based information. 
%The posterior distributions of the degradation signal model parameters can be used to estimate the RLD of each component. 
For that purpose, we generate $100$ unique degradation signals with a random initial signal amplitude. For the sake of easier modeling, we assume that these degradation signals were observed at some random discrete times. We further assume that random observation time $t^k_h$ for component $h$ follows a uniform distribution on $[1, (\Lambda-\mu_0)/(\mu_1 + 3\kappa_1)]$. This assumption implies that degradation signal for component $h \in \mathcal{H}$ was observed when it had been drastically degrading with a gradual linear drift. Under these assumptions, we obtain the posterior mean of the drift parameter $\beta_h$ of form \eqref{posteriorDrift}, which easily yields us to identify the RLD of each component $h \in \mathcal{H}$ (Proposition \ref{prop:inverseGauss}). Consequently, we select set $\mathcal{H'}$ by means of RLDs as discussed in Section \ref{ss:problemSetting}.

\subsubsection{Computational Setup}
 The code for each algorithm is written in Python using Spyder IDE. We use a 64-bit computer with Intel Xeon W-2255 CPU with a 2.20 GHz processor and 32 GB of memory space, running on the Windows operating system. The Gurobi Optimizer (\texttt{GUROBI}) is used to solve the pure binary integer first-stage problem \eqref{eq:twoStageCompact} and the mixed-integer operational subproblems \eqref{eq:recourseFunction}. To benefit from the decomposition of the operational subproblems throughout the implementation, we employ Joblib library for parallel computing. We use PoissonBinomial library \citep{pypi} as our probability oracle. We allow \texttt{GUROBI} to use 20 threads for solving \eqref{eq:twoStageCompact}, however, we set the number of parallel threads parameter \texttt{Threads} to 1 for solving \eqref{eq:recourseFunction} when using parallelization. We use the relative optimality gap, $\% (UB-LB)/UB$, as a stopping criteria within Algorithm \ref{alg:decomposition}. For each computational experiment, the relative optimality gap tolerance \texttt{MIPGap} is chosen as the same as the tolerance parameter $\epsilon$ of Algorithm \ref{alg:decomposition}. Time limit for all experiments is set to 6 hours. The time for computational experiments is measured in seconds. Note that each operational subproblem is solved to optimality within the tolerance $\epsilon$.
%\subsection{Computational Results} \label{ss:ComputationalResults}
%In this section, we present the computational results of our extensive computational study.

\subsection{Performance of the Proposed Algorithm} \label{ss:performanceOfTheProposedAlgorithm}
In this section, we illustrate the computational efficiency of the proposed algorithm from three aspects. We first benchmark the performance of the proposed optimality cuts and algorithmic enhancements against the standard integer L-shaped optimality cut and the state-of-the-art solver \texttt{GUROBI} for different sizes of failure scenarios under exact and approximate representations of the joint chance-constraint. Secondly, we present the speedup of the proposed algorithm gained from parallel computing. We conclude this section by comparing the performance of each $FlowModel$ used for transmission line flow analysis over different instances.

\subsubsection{Benchmark of the Proposed Algorithm}
We derive different sets of optimality cuts based on the integer L-shaped optimality cuts (Section \ref{s:optimalityCutFamilies}) and introduce various algorithmic enhancements such as time-decomposability of scenario subproblems and status of system components (Section \ref{ss:decompositionAlgorithm}). By using both the exact representation (referred to as SP$_{\tiny \text{exact}}$) and the SOCP-based safe approximation (referred to as SP$_{\tiny \text{safe}}$) of the joint chance-constraint, we compare them against each other over the illustrative $9$-bus instance under failure scenarios of size $50, 100$ and $200$. When using \texttt{GUROBI} for SP$_{\tiny \text{exact}}$, we first relax our optimization model by removing the joint chance-constraint and obtain a solution within the time limit. Then, we check the feasibility status of this solution with respect to the joint chance-constraint with Algorithm \ref{alg:representChance}. When this solution is feasible, we conclude that it is indeed optimal. Otherwise, we discard this solution from the set of feasible solutions by adding \eqref{eq:extendedCoverIneq} to our optimization model and resolve it by \texttt{GUROBI}. We investigate the differences between these optimality cuts and algorithmic enhancements by setting the tolerance parameter $\epsilon$ to $10^{-2}$. Our computational results are shown in Table \ref{table:benchmark} for the following cases of Algorithm \ref{alg:decomposition}:
    \begin{itemize}
        \item \texttt{intLS}: The set of classical integer L-shaped optimality cuts in \eqref{eq:classical_Lshaped_single}.
        \item \texttt{optCut}: The set of improved optimality cuts in  \eqref{eq:eq:optCut_K+_single}.
        \item \texttt{optCut}$_{\footnotesize+}$: The set of improved optimality cuts in \eqref{eq:eq:optCut_K++_single}. 
        \item \texttt{intLS}$^*$: The set of classical integer L-shaped optimality cuts in \eqref{eq:classical_Lshaped_single} with time-decomposability of scenario subproblems and status of system components. 
        \item  \texttt{optCut}$^*$: The set of improved optimality cuts in  \eqref{eq:eq:optCut_K+_single} with time-decomposability of scenario subproblems and status of system components.
        \item \texttt{optCut}$^*_{\footnotesize ++}$: The set of improved optimality cuts in \eqref{eq:optCut_KT+++} with time-decomposability of scenario subproblems and status of system components.
        \item \texttt{GUROBI}: The state-of-the-art solver \texttt{GUROBI}.
    \end{itemize}

\begin{table}
\begin{center}{\scalebox{1}{
\begin{tabular}{lcrrrrrrrll}
\hline
\multirow{2}{*}{}     & \multicolumn{1}{c}{\multirow{2}{*}{$|\mathcal{K}|$}} & \multicolumn{1}{c}{\multirow{2}{*}{\texttt{intLS}}} & \multicolumn{1}{c}{\multirow{2}{*}{\texttt{optCut}}} & \multicolumn{1}{c}{\multirow{2}{*}{\texttt{optCut}$_{\footnotesize+}$}} & \multicolumn{1}{l}{\multirow{2}{*}{\texttt{intLS}$^*$}} & \multicolumn{1}{c}{\multirow{2}{*}{\texttt{optCut}$^*$}} & \multicolumn{1}{c}{\multirow{2}{*}{\texttt{optCut}$^*_{\footnotesize ++}$}} & \multicolumn{1}{c}{\multirow{2}{*}{\texttt{GUROBI}}} &  & \multicolumn{1}{c}{\multirow{2}{*}{Speedup}} \\
                      & \multicolumn{1}{l}{}                                 & \multicolumn{1}{l}{}                         & \multicolumn{1}{l}{}                          & \multicolumn{1}{l}{}                            & \multicolumn{1}{l}{}                           & \multicolumn{1}{l}{}                            & \multicolumn{1}{l}{}                                                  & \multicolumn{1}{l}{}                                 &  & \multicolumn{1}{l}{}                         \\ \hline
\multirow{6}{*}{{SP$_{\tiny \text{exact}}$}}  & \multirow{2}{*}{50}                                  & \multirow{2}{*}{2222.29}                     & \multirow{2}{*}{2139.21}                      & \multirow{2}{*}{871.66}                         & \multirow{2}{*}{232.13}                        & \multirow{2}{*}{123.04}                         & \multirow{2}{*}{11.29}                                              & \multirow{2}{*}{2576.07}                               &  & \multirow{2}{*}{$\times 228.27$}                     \\
                      &                                                      &                                              &                                               &                                                 &                                                &                                                 &                                                                       &                                                      &  &                                              \\
                      & \multirow{2}{*}{100}                                 & \multirow{2}{*}{4184.08}                     & \multirow{2}{*}{4025.15}                      & \multirow{2}{*}{1884.74}                        & \multirow{2}{*}{400.89}                        & \multirow{2}{*}{256.71}                         & \multirow{2}{*}{18.00}                                              & \multirow{2}{*}{5311.97}                               &  & \multirow{2}{*}{$\times 295.04$}                     \\
                      &                                                      &                                              &                                               &                                                 &                                                &                                                 &                                                                       &                                                      &  &                                              \\
                      & \multirow{2}{*}{200}                                 & \multirow{2}{*}{8165.64}                     & \multirow{2}{*}{7971.17}                      & \multirow{2}{*}{4235.76}                        & \multirow{2}{*}{772.62}                        & \multirow{2}{*}{604.62}                         & \multirow{2}{*}{29.30}                                             & \multirow{2}{*}{16488.76}                               &  & \multirow{2}{*}{$\times 562.73$}                     \\
                      &                                                      &                                              &                                               &                                                 &                                                &                                                 &                                                                       &                                                      &  &                                              \\ \hline
\multirow{6}{*}{{SP$_{\tiny \text{safe}}$}} & \multirow{2}{*}{50}                                  & \multirow{2}{*}{334.07}                      & \multirow{2}{*}{334.70}                       & \multirow{2}{*}{350.69}                         & \multirow{2}{*}{11.85}                         & \multirow{2}{*}{12.36}                          & \multirow{2}{*}{9.30}                                               & \multirow{2}{*}{576.33}                                &  & \multirow{2}{*}{$\times 61.99$}                      \\
                      &                                                      &                                              &                                               &                                                 &                                                &                                                 &                                                                       &                                                      &  &                                              \\
                      & \multirow{2}{*}{100}                                 & \multirow{2}{*}{639.34}                      & \multirow{2}{*}{637.40}                       & \multirow{2}{*}{766.17}                         & \multirow{2}{*}{18.63}                         & \multirow{2}{*}{17.83}                          & \multirow{2}{*}{12.36}                                              & \multirow{2}{*}{2007.75}                               &  & \multirow{2}{*}{$\times 162.45$}                     \\
                      &                                                      &                                              &                                               &                                                 &                                                &                                                 &                                                                       &                                                      &  &                                              \\
                      & \multirow{2}{*}{200}                                 & \multirow{2}{*}{1268.76}                     & \multirow{2}{*}{1270.58}                      & \multirow{2}{*}{1879.29}                        & \multirow{2}{*}{28.23}                         & \multirow{2}{*}{29.63}                          & \multirow{2}{*}{19.02}                                              & \multirow{2}{*}{6155.00}                               &  & \multirow{2}{*}{$\times 323.65$}                     \\
                      &                                                      &                                              &                                               &                                                 &                                                &                                                 &                                                                       &                                                      &  &                                              \\ \hline
\end{tabular}
}
}
    \end{center}
    \caption{Computational Times for the $9$-bus Instance.}
\label{table:benchmark}
\end{table}

The ``Speedup'' column represents the speedup of \texttt{optCut}$^*_{\footnotesize ++}$ against \texttt{GUROBI}. For each scenario size, \texttt{GUROBI} is able to provide a feasible solution within the time limit; however, its computational time is even larger than \texttt{intLS}. For SP$_{\tiny \text{exact}}$, the computational times under \texttt{intLS} and \texttt{optCut} increase linearly with the size of scenarios whereas \texttt{optCut}$_{\footnotesize+}$ reduces these computational times almost by half. Surprisingly for SP$_{\tiny \text{safe}}$, \texttt{intLS} and \texttt{optCut} outperform \texttt{optCut}$_{\footnotesize+}$ under different sets of scenarios. The time-decomposability of scenario subproblems and status of system components provide the most computational gain in both SP$_{\tiny \text{exact}}$ and SP$_{\tiny \text{safe}}$ as these algorithmic enhancements prevent many unnecessary resolves of scenario subproblems within Algorithm \ref{alg:decomposition}. Under these enhancements, \texttt{optCut}$^*_{\footnotesize ++}$ outperforms  \texttt{intLS}$^*$ and \texttt{optCut}$^*$ for both exact and safe approaches. We observe that speedup gains for SP$_{\tiny \text{exact}}$ are more than for SP$_{\tiny \text{safe}}$ as the feasible region induced by the joint chance-constraint is smaller in the latter, which reduces the effects of the optimality cuts. Still, \texttt{optCut}$^*_{\footnotesize ++}$ provides a substantial speedup compared to \texttt{GUROBI} for both SP$_{\tiny \text{exact}}$ and SP$_{\tiny \text{safe}}$.

%Our computational study indicates that \texttt{optCut}$^*_{\footnotesize ++}$ is the most efficient compared to the remaining for $9$-bus instance under both SP$_{\tiny \text{exact}}$ and SP$_{\tiny \text{safe}}$.
Based on our preliminary computations of the SAA method, we observe that we can obtain maintenance and operational schedules within $2\%$ optimality under failure scenarios of size $50$ and $100$ (see Section \ref{ss:SAAResults}). We extend our computational study for all IEEE instances by setting the tolerance parameter $\epsilon$ to $10^{-4}$. In Table \ref{table:benchmark_}, we investigate the computational efficiency of \texttt{optCut}$^*_{\footnotesize ++}$  against \texttt{GUROBI} by reporting the following metrics:
\begin{itemize}
    \item \# Iter: The number of iterations within Algorithm \ref{alg:decomposition}.
    \item Time: The time for solving the joint chance-constrained stochastic program in seconds.
    \item Gap: The percentage relative optimality gap obtained within the $6$-hour time limit.
\end{itemize}
\begin{table}[H]
\begin{center}{\scalebox{1}{
\begin{tabular}{rcccccccccccccc}
\hline
                           &                                  & \multicolumn{6}{c}{\multirow{2}{*}{SP$_{\tiny \text{exact}}$}}                                                                                                               &  & \multicolumn{6}{c}{\multirow{2}{*}{SP$_{\tiny \text{safe}}$}}                                                                                                                                \\
                           &                                  & \multicolumn{6}{c}{}                                                                                                                                    &  & \multicolumn{6}{c}{}                                                                                                                                                     \\ \cline{3-8} \cline{10-15} 
                           &                                  & \multicolumn{3}{c}{\multirow{2}{*}{\texttt{optCut}$_{\footnotesize ++}^*$}}                                  &  & \multicolumn{2}{c}{\multirow{2}{*}{\texttt{GUROBI}}} &  & \multicolumn{3}{c}{\multirow{2}{*}{\texttt{optCut}$_{\footnotesize ++}^*$}}                                  & \multirow{2}{*}{} & \multicolumn{2}{c}{\multirow{2}{*}{\texttt{GUROBI}}} \\
                           &                                  & \multicolumn{3}{c}{}                                                         &  & \multicolumn{2}{c}{}                                                  &  & \multicolumn{3}{c}{}                                                         &                   & \multicolumn{2}{c}{}                                                  \\ \cline{3-5} \cline{7-8} \cline{10-12} \cline{14-15} 
                           & \multirow{2}{*}{$|\mathcal{K}|$} & \multirow{2}{*}{\# Iter} & \multirow{2}{*}{Time} & \multirow{2}{*}{Gap} &  & \multirow{2}{*}{Time}           & \multirow{2}{*}{Gap}           &  & \multirow{2}{*}{\# Iter} & \multirow{2}{*}{Time} & \multirow{2}{*}{Gap} &                   & \multirow{2}{*}{Time}           & \multirow{2}{*}{Gap}           \\
                           &                                  &                          &                       &                           &  &                                 &                                     &  &                          &                       &                           &                   &                                 &                                     \\ \hline
\multirow{2}{*}{$9$-bus}   & 50                               & 30                       & 11.29                 & 0.00                      &  & TL                               & 0.08                                  &  & 21                       & 10.02                 & 0.00                      &                   & 736.25                          & 0.00                                \\
                           & 100                              & 31                       & 18.00                 & 0.00                      &  & TL                               & 0.12                                   &  & 21                       & 13.30                 & 0.00                      &                   & 2931.43                         & 0.01                                \\ \hline
\multirow{2}{*}{$39$-bus}  & 50                               & 311                      & 3610.20               & 0.00                      &  & TL                               & 6.44                                &  & 75                       & 506.32                & 0.00                      &                   & TL                               & 0.72                                   \\
                           & 100                              & 348                      & 7334.40               & 0.00                      &  & TL                               & 30.63                               &  & 73                       & 598.75                & 0.00                      &                   & TL                               & 7.59                                   \\ \hline
\multirow{2}{*}{$57$-bus}  & 50                               & 329                      & 2509.55

             & 0.00                      &  & TL                               & 0.24                                &  & 393                      & 2110.02

            & 0.00                      &                   & TL                               & 0.07                                \\
                           & 100                              & 364                      &7708.44              & 0.01                      &  & TL                               & 0.29                                &  & 386                      & 3889.33              & 0.01                      &                   & TL                               & 0.16                                \\ \hline
\multirow{2}{*}{$118$-bus} & 50                               & 90                       & TL                    & 3.80                      &  & TL                               & 69.88                               &  & 699                      & TL                     & 0.79                      &                   & TL                               & 5.33                                   \\
                           & 100                              & 78                       & TL                    & 4.35                      &  & TL                               & NA                                  &  & 580                      & TL                     & 1.55                      &                   & TL                               & NA                                \\ \hline
\end{tabular}}}
    \end{center}
    \caption{Comparison of \texttt{optCut}$_{\footnotesize ++}^*$ with \texttt{GUROBI} for Different Instances.}
\label{table:benchmark_}
\end{table}
The ``TL'' (under column ``Time'') is used whenever the $6$-hour time limit is reached. The ``NA'' (under column ``Gap'') is used if no feasible solution is found within the time limit. According to Table \ref{table:benchmark_},  \texttt{optCut}$^*_{\footnotesize ++}$ and \texttt{GUROBI} produce an optimal solution within the time limit for the $9$-bus instance under SP$_{\tiny \text{safe}}$; however, \texttt{optCut}$^*_{\footnotesize ++}$ attains these solutions in less than $20$ seconds whereas the computational time of \texttt{GUROBI} rapidly increases when $100$ failure scenarios are used. For all instances, \texttt{optCut}$^*_{\footnotesize ++}$ outperforms \texttt{GUROBI} in terms of the percentage relative optimality gap. For the $118$-bus instance, \texttt{GUROBI} fails to produce a feasible solution within the time limit under scenario size of $100$ whereas \texttt{optCut}$^*_{\footnotesize ++}$ produces a feasible solution for both SP$_{\tiny \text{exact}}$ and SP$_{\tiny \text{safe}}$. Our computational study shows that \texttt{optCut}$^*_{\footnotesize ++}$ has significant computational gains compared to \texttt{GUROBI} under both SP$_{\tiny \text{exact}}$ and SP$_{\tiny \text{safe}}$, and can be used to produce high-quality feasible solutions for large-scale instances. 

\subsubsection{Parallel Computing}
A significant property of Algorithm \ref{alg:decomposition} is that the linear relaxations (Step \ref{alg:decompInitial}) and the second-stage problems (Step \ref{alg:decompStatusParallel}) can be solved in parallel. 
%In our computational experiments, we benefit from this property to achieve the most in terms of computational efficiency. 
In order to demonstrate the effect of parallelism within Algorithm \ref{alg:decomposition}, we solve the $9$-bus instance with a scenario size of $1000$ by using the exact reformulation of the joint chance-constraint. The results of our computational experiment are presented in Figure \ref{fig:parallelComputing} with respect to different number of threads.

\begin{figure}[H]
\centering
\begin{tikzpicture}[scale=1]
\draw[gray, dashdotted, xstep=1.555cm, ystep=1.29cm] (0,0) grid (6.9,5.7);
\begin{axis}[
tick align=outside,
tick pos=left,
xtick={0,1,2,3,4},
xticklabels={$2^0$,$2^1$,$2^2$,$2^3$,$2^4$},
ytick={0,1,2,3,4},
yticklabels={$2^0$,$2^1$,$2^2$,$2^3$,$2^4$},
%x grid style={white!69.0196078431373!black},
xlabel={Number of threads},
xtick style={color=black},
y grid style={white!69.0196078431373!black},
ymin ={0},
xmin={0},
ylabel={Speedup},
ytick style={color=black},
legend pos=north west,
legend style={nodes={scale=0.7, transform shape}}
]
\addplot[semithick, blue, mark=square]
table[row sep=crcr] {%
0 0\\
1 0.83\\
2 1.49\\
3 2.05\\
4 2.15\\
};
\addlegendentry{Observed speedup};
\addlegendentry{Ideal speedup};
\addplot[semithick, black, mark=square]
table[row sep=crcr] {%
0 0\\
1 1\\
2 2\\
3 3\\
4 4\\
};
\addlegendentry{Ideal speedup};
\end{axis}
\end{tikzpicture}
\caption{Speedup ratios with parallel computing.}
\label{fig:parallelComputing}
\end{figure}
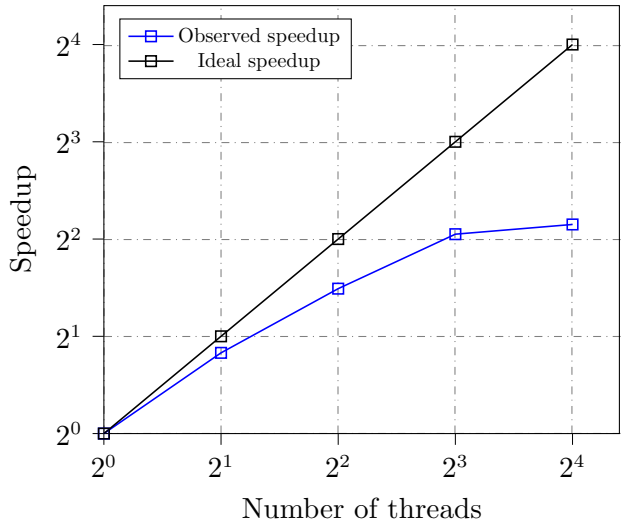
Our empirical study indicates a sublinear growth in the parallel speedup ratios. We note that the data size of $9$-bus instance is relatively small and the results are only representative, nevertheless, the utilization of the parallel computing becomes more apparent as the size of the problem increases.

\subsubsection{Flow Limit Analysis}
As a preprocessing step for Algorithm \ref{alg:decomposition}, we identify the redundancy in constraints \eqref{flowBounds} as explained in Section \ref{s:flowLimitAnalysis}. We summarize the computational results in Table \ref{table:RedundancyRatios} by reporting the following metrics:
    \begin{itemize}
        \item Time: The preprocessing time in seconds.
        \item UB: The redundancy ratio in the upper bound flow constraints.
        \item LB: The redundancy ratio in the  lower bound flow constraints.
    \end{itemize}
%The procedure for identifying flow limit redundancy differs in each $FlowModel$. For example, suppose that \textit{FlowModel-I} identifies redundancy only for transmission line $(i',j') \in \mathcal{L''}$ which implies that all of the corresponding flow upper (or similarly lower) bounds on transmission line $(i',j')$ within the planning horizon are redundant, therefore redundancy ratio given by \textit{FlowModel-I} is $1/|\mathcal{L''}|$. If \textit{FlowModel-I} identifies redundancy in upper (or lower) bounds for a different transmission line $(i'',j'') \in \mathcal{L''}$, this ratio is then given by $2/|\mathcal{L''}|$. 
We note that we do not report $57$-bus instance in Table \ref{table:RedundancyRatios} since all upper and lower flow limits are redundant. The redundancy ratios under UB and LB columns are given as follows:
\begin{itemize}
                \item The redundancy ratio over $|\mathcal{L''}|$ in \textit{FlowModel-I}.
                \item The redundancy ratio over $|\mathcal{L''}| \times |\mathcal{T}|$ in \textit{FlowModel-II}.
                \item The redundancy ratio over $|\mathcal{L''}| \times |\mathcal{T}| \times |\mathcal{S}|$ in \textit{FlowModel-III}.
    \end{itemize}
    
\begin{table}[H] 
        \begin{center}{\scalebox{1}{
\begin{tabular}{|lccccccccccc|}
\hline
& \multicolumn{3}{c}{\multirow{2}{*}{$9$-bus}}    &                      & \multicolumn{3}{c}{\multirow{2}{*}{$39$-bus}}                               &                      & \multicolumn{3}{c|}{\multirow{2}{*}{$118$-bus}}                             \\
& \multicolumn{3}{c}{}                                                        & \multicolumn{1}{l}{} & \multicolumn{3}{c}{}                                                        & \multicolumn{1}{l}{} & \multicolumn{3}{c|}{}                                                       \\ \cline{2-4} \cline{6-8} \cline{10-12} 
 & \multirow{2}{*}{UB} & \multirow{2}{*}{LB} & \multirow{2}{*}{Time} &                      & \multirow{2}{*}{UB} & \multirow{2}{*}{LB} & \multirow{2}{*}{Time} &                      & \multirow{2}{*}{UB} & \multirow{2}{*}{LB} & \multirow{2}{*}{Time} \\
&   &   &   & \multicolumn{1}{l}{} & &   &  & \multicolumn{1}{l}{} &    &   &  \\ \hline
\textit{FlowModel-I}   & 0.500   & 0.333  & 0.011  &   & 0.476  & 0.548  & 0.114 & & 0.819 & 0.819 & 3.541 \\
\textit{FlowModel-II}  & 0.500   & 0.476  & 0.035  &   & 0.514  & 0.548  & 0.740 & & 0.819 & 0.822 & 24.432\\
\textit{FlowModel-III} & 0.602   & 0.640  & 0.593  &   & 0.560  & 0.548  & 17.650& & 0.831 & 0.832 & 574.311\\ \hline
\end{tabular}}}
    \end{center}
    \caption{Flow Limit Analysis.}
    \label{table:RedundancyRatios}
    \end{table}
Each {\it{FlowModel}} identifies redundacy less than a second for the $9$-bus instance whereas the differences between the preprocessing times become more evident as the instance size increases. For $39$-bus and $118$-bus instances, the differences between the redundancy ratios given by {\it FlowModel-I} and {\it FlowModel-II} remain almost identical. Thus, one can potentially consider the trade-off between the computational effort and redundancy in the choice of \textit{FlowModel}. As instance size increases, the difference between the redundancy ratios tends to decrease. Still, our computational results show that \textit{FlowModel-III} provides the best ratio within a reasonable time limit for all instances. Therefore, we use  \textit{FlowModel-III} to identify such redundant flow limits in the remainder of our computational experiments.

\subsection{Sample Average Approximation Results} \label{ss:SAAResults}
In this section, we present our computational results by solving the SAA problems of the joint chance-constrained stochastic program by evaluating the obtained solutions through Algorithm \ref{alg:SAA}. For that purpose, we let $M = 5$, $N' = 1000$, $N = 50$ and $N =100$ with significance level of $0.05$ of the SAA algorithm. We generate i.i.d. samples for each replicate and solve them with Algorithm \ref{alg:decomposition} for various IEEE test instances. We remind the readers that the training scenarios are generated over the set of components $\mathcal{H'}$ whereas solutions are evaluated over the failure possibilities of all components $\mathcal{H}$. The resulting $95\%$ CIs for the lower and upper bound estimates (in $100.000 \$$) are presented in Table \ref{table:SAAResults}. We also report the estimated optimality gaps between the optimal value associated with the candidate optimal solutions produced by the SAA method and the true optimal value in Table \ref{table:SAAResults}.

\begin{table}[H]
\begin{center}{\scalebox{1}{

\begin{tabular}{rcccccccc}
\hline
\multicolumn{1}{l}{}     & \multicolumn{1}{l}{}             & \multicolumn{3}{c}{\multirow{2}{*}{{SP$_{\tiny \text{exact}}$}}}                          & \multicolumn{1}{l}{} & \multicolumn{3}{c}{\multirow{2}{*}{{SP$_{\tiny \text{safe}}$}}}                           \\
                         & \multicolumn{1}{l}{}             & \multicolumn{3}{c}{}                                                              & \multicolumn{1}{l}{} & \multicolumn{3}{c}{}                                                              \\ \cline{3-5} \cline{7-9} 
\multicolumn{1}{c}{}    & \multirow{2}{*}{$|\mathcal{K}|$} & \multirow{2}{*}{CI of LB} & \multirow{2}{*}{CI of UB} & \multirow{2}{*}{Gap (\%)} &                      & \multirow{2}{*}{CI of LB} & \multirow{2}{*}{CI of UB} & \multirow{2}{*}{Gap (\%)} \\
                         &                                  &                           &                           &                           &                      &                           &                           &                           \\ \hline
\multirow{2}{*}{9-bus}   & 50                               & (1.61, 1.64)              & (1.63, 1.64)              & 2.30                      &                      & (1.64, 1.67)              & (1.66, 1.67)              & 2.03                      \\
                         & 100                              & (1.63, 1.64)              & (1.63, 1.64)              & 1.17                      &                      & (1.64, 1.66)             & (1.66, 1.67)              & 1.65                      \\ \hline
\multirow{2}{*}{39-bus}  & 50                               & (31.50, 31.80)            & (31.69, 31.88)            & 1.21                      &                      & (36.29, 36.41)            & (36.39, 36.47)            & 0.48                      \\
                         & 100                              & (31.65, 31.88)            & (31.69, 31.89)            & 0.72                      &                      & (36.32, 36.38)            & (36.39, 36.47)            & 0.41                      \\ \hline
\multirow{2}{*}{57-bus}  & 50                               & (36.63, 36.79)            & (36.91, 37.02)            & 1.07                      &                      & (36.72, 36.83)            & (36.97, 37.08)            & 0.98                      \\
                         & 100                              & (36.67, 36.78)            & (36.90, 37.02)            & 0.95                      &                      & (36.74, 36.85)            & (36.97, 37.08)            & 0.93                      \\ \hline
\multirow{2}{*}{118-bus} & 50                               & (5.48, 5.52)              & (5.53, 5.55)               & 1.22      
            &                      & (5.49, 5.52)    & (5.53, 5.55)     & 1.12      \\
                         & 100                              & (5.50, 5.52)     & (5.53, 5.55)      & 0.86 
            &                      & (5.49, 5.51)      & (5.53, 5.55)      & 1.10      \\ \hline
\end{tabular}

}

}
\end{center}
\caption{SAA Results.}
\label{table:SAAResults}\end{table}

According to Table \ref{table:SAAResults}, the estimated gap decreases as the size of scenarios increases for all instances, as expected. For the $9$-bus instance, the scenario size of $100$ reduces the estimated gap significantly compared to the scenario size of $50$ whereas this reduction is less significant in other instances. %Still, this difference between the estimated gaps under different scenario sizes gets smaller for larger instances. 
Additionally, the estimated confidence intervals are almost identical for the $118$-bus instance for both SP$_{\tiny \text{exact}}$ and SP$_{\tiny \text{safe}}$. The results of our computational study indicate that these sample sizes along with the choice of the system components considered for maintenance are indeed large enough to obtain the corresponding tight bounds on the true optimal value of our optimization model. 
    
\subsection{Model Comparison} \label{ss:modelComparison}
In this section, we evaluate the quality of the maintenance schedules obtained from the proposed stochastic models, SP$_{\tiny \text{exact}}$ and SP$_{\tiny \text{safe}}$, in terms of the average failures of system components, maintenance and operational costs under $50$ failure scenarios. In order to quantify the effects of these schedules when the unexpected failures are not considered, we compare the maintenance schedules of SP$_{\tiny \text{exact}}$ and SP$_{\tiny \text{safe}}$ with those of a deterministic model (DM), which assumes that  none of the system components will fail within the planning horizon. We evaluate each maintenance schedule over failure scenarios of size $1000$ and report the average failures in Table \ref{table:avgFailures}. 
\begin{table}[H]
\scriptsize
\centering
\begin{tabular}{rclccclccclccclccc}
\hline
&&& \multicolumn{3}{c}{\multirow{2}{*}{$\mathcal{G'}$}}  & & \multicolumn{3}{c}{\multirow{2}{*}{$\mathcal{L'}$}}
 && \multicolumn{3}{c}{\multirow{2}{*}{$\mathcal{G''} \cup \mathcal{L}''$}}  & & \multicolumn{3}{c}{\multirow{2}{*}{JCC-Violation}}\\
        &                                            &  & \multicolumn{3}{c}{}                                                     &  & \multicolumn{3}{c}{}                                                     &  & \multicolumn{3}{c}{}                                                     &  & \multicolumn{2}{c}{}                                      \\ \cline{4-6} \cline{8-10} \cline{12-14} \cline{16-18} 
        & \multirow{2}{*}{$(|\mathcal{G'}|, |\mathcal{L'}|)$} &  & \multirow{2}{*}{SP$_{\tiny \text{exact}}$} & \multirow{2}{*}{SP$_{\tiny \text{safe}}$} & \multirow{2}{*}{DM} &  & \multirow{2}{*}{SP$_{\tiny \text{exact}}$} & \multirow{2}{*}{SP$_{\tiny \text{safe}}$} & \multirow{2}{*}{DM} &  & \multirow{2}{*}{SP$_{\tiny \text{exact}}$} & \multirow{2}{*}{SP$_{\tiny \text{safe}}$} & \multirow{2}{*}{DM} &  & \multirow{2}{*}{SP$_{\tiny \text{exact}}$} & \multirow{2}{*}{SP$_{\tiny \text{safe}}$} & \multirow{2}{*}{DM}   \\
        &   &  &   &                          &                     &  &                         &                          &                     &  &                         &                          &                     &  &                             &                            \\ \hline
9-bus   & $(1,3)$                                             &  & 0.014                   & 0.014                    & 1.000               &  & 0.254                   & 0.000                    & 1.390               &  & \multicolumn{3}{c}{0.043}                                                &  & 0.000                           & 0.000                        &0.417   \\ \hline
39-bus  & $(4, 4)$                                            &  & 0.559                   & 0.022                    & 3.075               &  & 0.335                   & 0.000                    & 3.950               &  & \multicolumn{3}{c}{0.095}                                                &  & 0.097                           & 0.003                        & 1.000  \\ \hline
57-bus  & $(2, 7)$                                            &  & 0.010                   & 0.010                    & 1.894               &  & 0.190                   & 0.000                    & 6.067               &  & \multicolumn{3}{c}{0.259}                                                &  & 0.000                           & 0.000                      & 0.999     \\ \hline
118-bus & $(4, 9)$                                            &  & 0.064                   & 0.062                    & 3.798               &  & 0.340                   & 0.091                    & 6.776               &  & \multicolumn{3}{c}{0.124}                                                &  & 0.006                           & 0.006      & 1.000                     \\ \hline
\end{tabular}

\caption{Average Failures under Stochastic and Deterministic Models.}
\label{table:avgFailures}\end{table}
The ``JCC-Violation'' column represents the total number of joint chance-constraint violations under different maintenance plans by evaluating the number of components entering corrective maintenance under each scenario against the desired thresholds. For all instances, these violations are less than the probability level of the joint chance-constraint for both SP$_{\tiny \text{exact}}$ and SP$_{\tiny \text{safe}}$; however, these are adversely higher under DM as it does not consider the risks associated with the unexpected failures. Furthermore, SP$_{\tiny \text{safe}}$ provides a more conservative approach with less number of failures and lower violation of the joint-chance constraint, compared to the SP$_{\tiny \text{exact}}$ approach. In Table \ref{table:costComparison}, we also report the maintenance and operational costs incurred under these different maintenance schedules. 
\begin{table}[H]
\scriptsize
\centering
\begin{tabular}{rccccccccccccccc}
\hline
        &                                                 & \multicolumn{3}{c}{\multirow{2}{*}{GM}}                              &  & \multicolumn{3}{c}{\multirow{2}{*}{TLM}}                             &  & \multicolumn{3}{c}{\multirow{2}{*}{Operations}}                      &  & \multicolumn{2}{c}{\multirow{2}{*}{Cost Improv. (\%)}} \\
        &                                                 & \multicolumn{3}{c}{}                                                 &  & \multicolumn{3}{c}{}                                                 &  & \multicolumn{3}{c}{}                                                 &  & \multicolumn{2}{c}{}                                     \\ \cline{3-5} \cline{7-9} \cline{11-13} \cline{15-16} 
        & \multirow{2}{*}{$(|\mathcal{G'}|, |\mathcal{L}'|)$} & \multirow{2}{*}{SP$_{\tiny \text{exact}}$} & \multirow{2}{*}{SP$_{\tiny \text{safe}}$} & \multirow{2}{*}{DM} &  & \multirow{2}{*}{SP$_{\tiny \text{exact}}$} & \multirow{2}{*}{SP$_{\tiny \text{safe}}$} & \multirow{2}{*}{DM} &  & \multirow{2}{*}{SP$_{\tiny \text{exact}}$} & \multirow{2}{*}{SP$_{\tiny \text{safe}}$} & \multirow{2}{*}{DM} &  & \multirow{2}{*}{SP$_{\tiny \text{exact}}$}      & \multirow{2}{*}{SP$_{\tiny \text{safe}}$}      \\
        &                                                 &                       &                        &                     &  &                       &                        &                     &  &                       &                        &                     &  &                            &                             \\ \hline
9-bus   & (1,3)                                           & 0.32                  & 0.32                   & 0.91                &  & 0.04                  & 0.05                   & 0.06                &  & 1.28                  & 1.30                   & 1.39                &  & 30.67                      & 29.52                       \\ \hline
39-bus  & (4,4)                                           & 2.73                  & 2.42                   & 5.88                &  & 0.26                  & 0.22                   & 0.64                &  & 28.80                 & 33.79                  & 35.91               &  & 25.08                      & 14.14                       \\ \hline
57-bus  & (2,7)                                           & 3.78                  & 3.78                   & 10.09               &  & 1.13                  & 1.20                   & 2.95                &  & 32.05                 & 32.05                  & 31.98               &  & 17.89                      & 17.74                       \\ \hline
118-bus & (4,9)                                           & 1.13                  & 1.12                   & 3.07                &  & 0.22                  & 0.22                   & 0.49                &  & 4.20                  & 4.19                   & 4.26                &  & 29.07                      & 29.05                       \\ \hline
\end{tabular}

\caption{Cost Comparison of Stochastic and Deterministic Models.}
\label{table:costComparison}\end{table}
The ``GM'' and ``TLM'' columns provide the generator and transmission line maintenance costs, respectively. The ``Operations'' column gives the operational costs. All costs are reported in $100.000 \$$. The ``Cost Improv. (\%)'' represents the total cost improvements in percentages achieved by stochastic models compared to the deterministic model.

Table \ref{table:avgFailures} shows that the average failures for $\mathcal{G'}$ and $\mathcal{L'}$ significantly decrease under both SP$_{\tiny \text{exact}}$ and SP$_{\tiny \text{safe}}$ as DM ignores the power system capabilities. Accordingly, generator and transmission line maintenance costs obtained under these stochastic methods are less than under those of DM for all instances. We also observe a slight increase in the operational costs in the DM approach, except the $57$-bus instance; however, DM still incurs a higher total cost than SP$_{\tiny \text{exact}}$ and SP$_{\tiny \text{safe}}$. This is due to the fact that the effects of the unexpected failures of system components on power system operations are ignored in DM. Our computational study shows that $14-31 \%$ cost savings can be obtained under stochastic models in comparison with DM. As a result, the coordination between maintenance and operational schedules when considering the unexpected failures of system components yields  significant cost savings as well as less interruptions due to these failures.

\subsection{Sensitivity Analysis} \label{ss:sensitivityAnalysis}
In this section, we examine the effects of different choices of sets $\mathcal{G'}$ and $\mathcal{L'}$ on average failures, maintenance and operational costs under $50$ failure scenarios. For that purpose, we first select $(p_{fail}^{\mathcal{G}}, p_{fail}^{\mathcal{L}}) = (0.2, 0.4)$ which decreases the cardinality of these sets compared to the baseline setting; however, this selection of subsets results in infeasibilities. This is because of the fact that components in $\mathcal{H''}$ are not scheduled for maintenance within the planning horizon, which causes the  violation of the joint chance-constraint. Then, we analyze the effects of the size of sets $\mathcal{G'}$ and $\mathcal{L'}$ when $(p_{fail}^{\mathcal{G}}, p_{fail}^{\mathcal{L}}) = (0.01, 0.02)$ that considers more components for maintenance. We evaluate the maintenance schedules obtained by stochastic and deterministic models over $1000$ failure scenarios, which are the same in Section \ref{ss:modelComparison}. We report the average failures and joint chance-constraint violations in Table \ref{table:avgFailuresBig}. 
\begin{table}[H]
\scriptsize
\centering
\begin{tabular}{rclccclccclccclccc}
\hline
        &   &  & \multicolumn{3}{c}{\multirow{2}{*}{$\mathcal{G'}$}}    &  & \multicolumn{3}{c}{\multirow{2}{*}{$\mathcal{L'}$}}                      &  & \multicolumn{3}{c}{\multirow{2}{*}{$\mathcal{G''} \cup \mathcal{L}''$}}  &  & \multicolumn{3}{c}{\multirow{2}{*}{JCC-Violation}} \\
        &                                                     &  & \multicolumn{3}{c}{}                                                     &  & \multicolumn{3}{c}{}                                                     &  & \multicolumn{3}{c}{}                                                     &  & \multicolumn{2}{c}{}                                      \\ \cline{4-6} \cline{8-10} \cline{12-14} \cline{16-18} 
        & \multirow{2}{*}{$(|\mathcal{G'}|, |\mathcal{L'}|)$} &  & \multirow{2}{*}{SP$_{\tiny \text{exact}}$} & \multirow{2}{*}{SP$_{\tiny \text{safe}}$} & \multirow{2}{*}{DM} &  & \multirow{2}{*}{SP$_{\tiny \text{exact}}$} & \multirow{2}{*}{SP$_{\tiny \text{safe}}$} & \multirow{2}{*}{DM} &  & \multirow{2}{*}{SP$_{\tiny \text{exact}}$} & \multirow{2}{*}{SP$_{\tiny \text{safe}}$} & \multirow{2}{*}{DM} &  & \multirow{2}{*}{SP$_{\tiny \text{exact}}$}     & \multirow{2}{*}{SP$_{\tiny \text{safe}}$} & \multirow{2}{*}{DM}      \\
        &                                                     &  &                         &                          &                     &  &                         &                          &                     &  &                         &                          &                     &  &                             &                             \\ \hline
9-bus   & $(2,3)$                                             &  & 0.057                   & 0.065                    & 1.043               &  & 0.254                   & 0.000                    & 1.390               &  & \multicolumn{3}{c}{0.000}                                                &  & 0.000                           & 0.000                        & 0.417   \\ \hline
39-bus  & $(7, 5)$                                            &  & 0.452                   & 0.065                    & 3.118               &  & 0.046                   & 0.046                    & 3.996              &  & \multicolumn{3}{c}{0.006}                                                &  & 0.062                           & 0.003                        & 1.000   \\ \hline
57-bus  & $(3, 11)$                                            &  & 0.041                   & 0.041                    & 1.925               &  & 0.381                   & 0.192                    & 6.259               &  & \multicolumn{3}{c}{0.036}                                                &  & 0.000                           & 0.000                     &0.999     \\ \hline
118-bus & $(5, 11)$                                            &  & 0.108                   & 0.089                    & 3.825               &  & 0.407                   & 0.144                    & 6.829               &  & \multicolumn{3}{c}{0.044}                                                &  & 0.008                           & 0.006                        &1.000   \\ \hline
\end{tabular}
\caption{Average Failures under Stochastic and Deterministic Models with Larger $\mathcal{H'}$.}
\label{table:avgFailuresBig} \end{table}

We observe that the average failures of set $\mathcal{G}'$ and $\mathcal{L'}$ increase; however, this is an expected result since more system components are under study for maintenance. Table \ref{table:costComparisonBig} demonstrates the maintenance and operational costs incurred when the failure probability thresholds are decreased.

\begin{table}[H]
\scriptsize
\centering
\begin{tabular}{rccccccccccccccc}
\hline
        &                                                 & \multicolumn{3}{c}{\multirow{2}{*}{GM}}                              &  & \multicolumn{3}{c}{\multirow{2}{*}{TLM}}                             &  & \multicolumn{3}{c}{\multirow{2}{*}{Operations}}                      &  & \multicolumn{2}{c}{\multirow{2}{*}{Cost Improv. (\%)}} \\
        &                                                 & \multicolumn{3}{c}{}                                                 &  & \multicolumn{3}{c}{}                                                 &  & \multicolumn{3}{c}{}                                                 &  & \multicolumn{2}{c}{}                                     \\ \cline{3-5} \cline{7-9} \cline{11-13} \cline{15-16} 
        & \multirow{2}{*}{$(|\mathcal{G'}|, |\mathcal{L}'|)$} & \multirow{2}{*}{SP$_{\tiny \text{exact}}$} & \multirow{2}{*}{SP$_{\tiny \text{safe}}$} & \multirow{2}{*}{DM} &  & \multirow{2}{*}{SP$_{\tiny \text{exact}}$} & \multirow{2}{*}{SP$_{\tiny \text{safe}}$} & \multirow{2}{*}{DM} &  & \multirow{2}{*}{SP$_{\tiny \text{exact}}$} & \multirow{2}{*}{SP$_{\tiny \text{safe}}$} & \multirow{2}{*}{DM} &  & \multirow{2}{*}{SP$_{\tiny \text{exact}}$}      & \multirow{2}{*}{SP$_{\tiny \text{safe}}$}      \\
        &                                                 &                       &                        &                     &  &                       &                        &                     &  &                       &                        &                     &  &                            &                             \\ \hline
9-bus   & (2,3)                                           & 0.32                  & 0.32                   & 0.91                &  & 0.04                  & 0.05                   & 0.06                &  & 1.28                  & 1.30                   & 1.39                &  & 30.67                      & 29.52                       \\ \hline
39-bus  & (7,5)                                           & 2.52                  & 2.42                   & 5.88                &  & 0.22                  & 0.22                   & 0.64                &  & 29.73                 & 33.79                  & 35.91               &  & 23.47                      & 14.14                       \\ \hline
57-bus  & (3,11)                                           & 3.78                  & 3.78                   & 10.09               &  & 1.13                  & 1.20                   & 2.95                &  & 32.05                 & 32.05                  & 31.98               &  & 17.89                      & 17.74                       \\ \hline
118-bus & (5,11)                                           & 1.13                  & 1.12                   & 3.07                &  & 0.22                  & 0.22                   & 0.49                &  & 4.20                  & 4.19                   & 4.26                &  & 28.98                      & 29.05                       \\ \hline
\end{tabular}

\caption{Cost Comparison of Stochastic and Deterministic Models with Larger $\mathcal{H'}$.}
\label{table:costComparisonBig}\end{table}
For $9$-bus and $57$-bus instances, increasing the sizes of sets $\mathcal{G'}$ and $\mathcal{L'}$ does not affect the quality of the maintenance schedules for both SP$_{\tiny \text{exact}}$ and SP$_{\tiny \text{safe}}$ as compared to the results in Section \ref{ss:modelComparison}. For the $39$-bus instance under SP$_{\tiny \text{exact}}$, we observe a slight decrease in both generator and transmission line maintenance costs whereas operational cost increases. On the other hand, there is a relatively small increase in maintenance and operational costs for the $118$-bus instance under SP$_{\tiny \text{exact}}$. This is because of the fact that large-scale instances cannot be solved to optimality within tolerance as increasing the size of $\mathcal{H'}$ increases the computational time required for convergence of the solution algorithm as well. Nevertheless in all cases, there are still significant cost savings compared to DM. We observe that although we take less failure risks by decreasing probability thresholds, we might be overly cautious which can result in higher operational costs. %Our computational analysis shows that considering components which are unlikely to fail within the planning horizon can affect the cost improvements achieved by stochastic models as solving the resulting problem becomes computationally more demanding when these sets are enlarged.

\section{Conclusions} \label{s:conclusions}
In this paper, we study a short-term condition-based integrated maintenance planning problem in coordination with the power system operations by considering the unexpected failures of generators as well as transmission lines. We formulate this problem as a two-stage joint chance-constrained stochastic program. Under a Bayesian setting, we obtain the RLDs of generators and transmission lines by using their degradation-based sensor information. We consider a specific subset of these components which are more prone to failure for scheduling maintenance and take the effects of their unexpected failures into account  based on their estimated RLDs. We introduce a joint chance-constraint to mitigate the failure risk in the power network by restricting the number of system components under corrective maintenance. We develop a decomposition algorithm by improving the integer L-shaped method with various algorithmic enhancements including derivation of stronger optimality cuts by exploiting the underlying problem structure. This algorithm also includes a separation subroutine to provide an exact representation of the joint chance-constraint by leveraging the Poisson Binomial random variables in this constraint. As an alternative approach, we also provide a SOCP-based safe approximation to represent the joint chance-constraint which provides computational advantages for larger scale instances, despite of its conservatism. Our computational experiments demonstrate the efficiency of the proposed decomposition algorithm along with the improved cut generation procedures and preprocessing steps which consistently outperforms the state-of-the-art solver for all test instances. Finally, we highlight that our proposed stochastic  models can obtain $14-31 \%$ cost savings against a deterministic model since maintenance and operational schedules are coordinated in these models while explicitly considering the effects of failure uncertainty on power system operations.

\renewcommand{\theHchapter}{A\arabic{chapter}}

\begin{APPENDICES} 

\section{The Monotonicity of Poisson Binomial Distribution.} \label{app:monotonicity}
We state Lemma \ref{lemma:monotonicity} which is used in the proof of Proposition \ref{prop:monotonicity}.

\begin{lemma} \label{lemma:monotonicity}
The cumulative distribution function of Poisson Binomial distribution is non-increasing with respect to success probability $p_i$ for all $i =1, \dots, n$. 
\end{lemma}
\begin{proof}[Proof.]
Let $Y$ be a Poisson Binomial random variable with success probabilities $p_1, \dots, p_n$. It suffices to show that the partial derivative of the cumulative distribution function of Poisson Binomial distribution with respect to $p_i$ is nonpositive for all $i = 1, \dots, n$. Without loss of generality, we concentrate on the $n$th Bernoulli random variable. 
The cumulative distribution function of Poisson Binomial distribution is given by:
\begin{align*} \label{eq:cdfPoissonBinomial}
    F(y, p_1, \dots, p_n) =\mathbb{P}(Y \le y) = \sum_{l = 0}^y f(l, p_1, \dots, p_n) = \sum_{l = 0}^y \sum_{A \in \mathcal{B}_l(1, \dots, n)} \prod_{i \in A} p_i \prod_{j \in A^c} (1-p_j),
\end{align*}
where $f(l, p_1, \dots, p_n)$ denotes its probability mass function, i.e., the probability of $l$ successes in $n$ Bernoulli trials, and $\mathcal{B}_l(1, \dots, n)$ denotes the set of all subsets of size $l$ from $\{1, \dots, n\}$. We can rewrite the probability mass function of Poisson Binomial distribution as follows: 
\begin{align*}
     f(y, p_1, \dots, p_n) &= \sum_{A \in \mathcal{B}_y(1, \dots, n): n \in A} \prod_{i \in A} p_i \prod_{j \in A^c} (1-p_j) + \sum_{A \in \mathcal{B}_y(1, \dots, n): n \notin A} \prod_{i \in A} p_i \prod_{j \in A^c}  (1-p_j)\\
   &=  p_n  \sum_{A \in \mathcal{B}_y(1, \dots, n): n \in A} \prod_{i \in A: i \neq n} p_i  \prod_{j \in A^c} (1-p_j) + (1-p_n) \sum_{A \in \mathcal{B}_y(1, \dots, n): n \notin A} \prod_{i \in A} p_i \prod_{j \in A^c: j \neq n}  (1-p_j) . 
\end{align*}
Let us now obtain the partial derivative of $f(y,p_1, \dots, p_n)$ with respect to $p_n$. In fact, we have:
\begin{align*}
   \frac{\partial f(y,p_1, \dots, p_n)}{\partial p_n} &=\sum_{A \in \mathcal{B}_y(1, \dots, n): n \in A} \prod_{i \in A: i \neq n} p_i  \prod_{j \in A^c} (1-p_j) - \sum_{A \in \mathcal{B}_y(1, \dots, n): n \notin A} \prod_{i \in A} p_i \prod_{j \in A^c: j \neq n}  (1-p_j)\\
   &=\sum_{A \in \mathcal{B}_{y-1}(1, \dots, n-1)} \prod_{i \in A} p_i  \prod_{j \in A^c} (1-p_j) - \sum_{A \in \mathcal{B}_y(1, \dots, n-1)} \prod_{i \in A} p_i \prod_{j \in A^c}  (1-p_j).
\end{align*}
Consider the quantity $ H^{y-1} := \sum_{A \in \mathcal{B}_{y-1}(1, \dots, n-1)} \prod_{i \in A} p_i  \prod_{j \in A^c} (1-p_j)$. The first term in the last equality follows from the fact that the index $n$ indeed belongs to set $\mathcal{B}_y(1, \dots, n)$, but is not used in any of the multiplication operations. This is equivalent to the selection of $y-1$ many elements from $\{1, \dots, n-1\}$. Similarly, consider the quantity $ H^{y} := \sum_{A \in \mathcal{B}_y(1, \dots, n-1)} \prod_{i \in A} p_i \prod_{j \in A^c}  (1-p_j)$. The second term in the last equality is due to the fact that the index $n$ does not belong to set $\mathcal{B}_y(1, \dots, n)$, and is not used in any of the multiplication operations. This is equivalent to the selection of $y$ many elements from $\{1, \dots, n-1\}$. Thus, the partial derivative of $f(y,p_1, \dots, p_n)$ with respect to $p_n$ is given by:
\[ \frac{\partial f(y,p_1, \dots, p_n)}{\partial p_n}  =\begin{cases}
- H^0 & \text{if } y = 0,\\
H^{y-1} - H^y &\text{if } 1 \le y \le n-1,\\
H^{n-1}  &\text{if } y = n.
\end{cases}
\]
Then, it is easy to obtain the partial derivative of $F(y,p_1, \dots, p_n)$ with respect to $p_n$ as follows:
\[ \frac{\partial F(y,p_1, \dots, p_n)}{\partial p_n}  =\begin{cases}
- H^0 & \text{if } y = 0,\\
 - H^y &\text{if } 1 \le y \le n-1,\\
0  &\text{if } y = n.
\end{cases}
\]
Since $p_n \in [0,1]$, we clearly have $ \frac{\partial F(y,p_1, \dots, p_n)}{\partial p_n}  \le 0$ for $y=0, \dots, n$. This proves the property of monotonicity of Poisson Binomial distribution with respect to $p_n$.
\Halmos
\end{proof}

\begin{proof}[Proof of Proposition \ref{prop:monotonicity}.]
Consider any pair of maintenance decisions $v' = (w', z'), v'' = (w'', z'')$ with the following property:  
\[(h, t') \le (h, t'') \text{ for } (h,t') \in \mathcal{I}(v'), \ (h,t'') \in \mathcal{I}(v'') \text{ and } h \in \mathcal{H'}.\]
Let us first consider the set of generators prone to failure. As before, we let $\hat \zeta_{\mathcal{G}}(w')$ and $\hat \zeta_{\mathcal{G}}(w'')$ be the Poisson Binomial random variables with success probabilities $\{p_i' = \mathbb{P}(\xi_i \le m_i(w')); \ i \in \mathcal{G}\}$ and $\{p_i'' = \mathbb{P}(\xi_i \le m_i(w'')); \ i \in \mathcal{G}\}$, respectively. Clearly, the maintenance schedule under decision $w''$ is in a later period than the maintenance schedule under decision $w'$, which implies that $m_i (w') \le m_i (w'')$ for $i \in \mathcal{G'}$. Then, we have $p_i' \le p_i''$. Secondly, we consider the set of transmission lines prone to failure. We let $\hat \zeta_{\mathcal{L}}(z')$ and $\hat \zeta_{\mathcal{L}}(z'')$ be the Poisson Binomial random variables with success probabilities $\{p_{ij}' = \mathbb{P}(\xi_{ij} \le m_{ij}(z')); \ (i,j) \in \mathcal{L}\}$ and $\{p_{ij}'' = \mathbb{P}(\xi_i \le m_i(z'')); \ i \in \mathcal{L}\}$, respectively. Similarly, we have $p_{ij}' \le p_{ij}''$ for $(i,j) \in \mathcal{L'}$.

By Lemma \ref{lemma:monotonicity}, we have $\mathbb{P}(\hat \zeta_{\mathcal{G}}(w') \le \rho_{\mathcal{G}}) \ge \mathbb{P}(\hat \zeta_{\mathcal{G}}(w'') \le \rho_{\mathcal{G}})$ and  $\mathbb{P}(\hat \zeta_{\mathcal{L}}(z') \le \rho_{\mathcal{L}}) \ge \mathbb{P}( \hat \zeta_{\mathcal{L}}(z'') \le \rho_{\mathcal{L}})$. By using the independence of these random variables, we immediately have that $ \mathcal{P}(v') \ge \mathcal{P}(v'')$.
\Halmos
\end{proof}

\section{SAA Algorithm.} \label{app:SAA}
\begin{algorithm}[H]
\caption{SAA}
\label{alg:SAA}
\begin{algorithmic}[1]
\STATE Generate an i.i.d. failure scenario sample of size $N'$ considering all system components $\mathcal{H}$.
\FORALL{$i = 1,\dots, M$}
\STATE Generate an i.i.d. failure scenario sample of size $N$ considering all system components $\mathcal{H'}$.
\STATE Solve $\hat z^i_N = \min \Big\{ \frac{1}{N}\sum_{k \in \mathcal{K}} \pi^k \Big(c_k^{\top} v + \sum_{t \in \mathcal{T}}  \mathcal{Q}_t(v, \xi_k) \Big):  v \in \mathcal{\hat V} \Big\}$ using Algorithm \ref{alg:decomposition} and obtain $\epsilon$-optimal solution $\hat v ^i_N.$
\STATE Evaluate $\hat v ^i_N$ over $N'$ scenarios: $\hat z^i_{N'}(\hat v ^i_N) = \frac{1}{N'}\sum_{k =1}^{N'} \pi^k(c_k^{\top} \hat v ^i_N + \sum_{t \in \mathcal{T}} \mathcal{Q}_t(\hat v ^i_N, \xi_k)).$
\ENDFOR
\vspace{1mm}
\STATE Select the best candidate solution $\hat v^* \in \argmin \{ \hat z_{N'}^1(\hat v ^1_N), \dots, \hat z_{N'}^M(\hat v ^M_N) \}$ and the best upper bound estimate $\hat \mu_{U} = \hat z_{N'}(\hat v^*)$.
\STATE Calculate the variance estimate of the true upper bound estimate:\\
\[\hat \sigma^2_{U} = \frac{1}{N'(N'-1)} \sum_{k = 1}^{N'} \bigg( \big(c_k^{\top} \hat v^*  + \sum_{t \in \mathcal{T}} \mathcal{Q}_t(\hat v^*, \xi_k)\big) - \hat \mu_U \bigg)^2.\]
\STATE Construct the approximate $(1-\alpha)$ level CI for the upper bound estimate as $ \hat \mu_U \pm z_{\alpha/2} \hat \sigma_U$. \label{alg:SAA_CIofUB}
\STATE Calculate the mean and variance estimates of the true lower bound estimate as $\hat \mu_L$ and $\hat \sigma_L^2$ as:
\[\hat \mu_L = \frac{1}{M} \sum_{i=1}^{M} \hat z^i_N \quad \text{ and } \quad \hat \sigma^2_{L} = \frac{1}{M(M-1)} \sum_{i = 1}^{M} \big( \hat z^i_N - \hat \mu_L\big)^2.\]
\STATE Construct the approximate $(1-\alpha)$ level CI for the lower bound estimate as $ \hat \mu_L \pm t_{\alpha/2, M-1} \hat \sigma_L$. \label{alg:SAA_CIofLB}
\STATE Construct the approximate $(1-\alpha)$ level CI for the true objective value as $(\hat \mu_L - t_{\alpha/2, M-1} \hat \sigma_L, \ \hat \mu_U + z_{\alpha/2} \hat \sigma_U)$.
\end{algorithmic}
\end{algorithm}

\end{APPENDICES}

\bibliographystyle{informs2014} 
\bibliography{references}

\end{document}